\documentclass[a4paper,12pt,oneside]{article}
\usepackage[english]{babel}
\usepackage[T1]{fontenc} 
\usepackage[utf8]{inputenc}
\usepackage{amsthm}
\usepackage{amsmath}
\usepackage{amssymb}  
\usepackage{indentfirst}
\usepackage{fancyhdr}
\usepackage{amsthm}
\usepackage{graphicx}
\usepackage{pdfpages}
\usepackage{esint}
\usepackage{url}

\definecolor{dg}{rgb}{0.01, 0.75, 0.24}

\numberwithin{equation}{section}

\theoremstyle{plain} 
\newtheorem{thm}{Theorem}[section]

\newtheorem{lem}[thm]{Lemma} 
\newtheorem{prop}[thm]{Proposition} 
\newtheorem{rmk}[thm]{Remark} 
\newtheorem{dfn}[thm]{Definition}

\newcommand{\R}{\numberset{R}}

\newcommand{\e}{\varepsilon}

\def\XXint#1#2#3{{\setbox0=\hbox{$#1{#2#3}{\int}$}
		\vcenter{\hbox{$#2#3$}}\kern-.5\wd0}}

\def\dx{dx}

\def\d\theta{{\mathrm d}\theta}
\def\R{\mathbb{R}}
\def\e{\varepsilon}
\def\dist{\mbox{dist}}

\usepackage{xcolor}                    
\definecolor{custom-blue}{RGB}{0,99,166} 
\usepackage{hyperref}
\hypersetup{colorlinks=true, allcolors=custom-blue}

\usepackage{geometry}
\allowdisplaybreaks[4]

\begin{document}

\title{\textbf{{Gradient regularity  for a class of elliptic obstacle problems}}}

\author { 
Raffaella Giova, 
Antonio Giuseppe Grimaldi,
%\thanks{Dipartimento di Scienze Matematiche "G. L. Lagrange",
%Politecnico di Torino, Corso Duca degli Abruzzi, 24, 10129 Torino, 
% Italy. E-mail: \textit{antonio.grimaldi@polito.it}} 
Andrea Torricelli
}

\maketitle
\maketitle

\begin{abstract}
We prove some regularity results for a priori bounded local minimizers of non-autonomous integral functionals of the form
$$\mathcal{F}(v,\Omega)=\int_\Omega F(x,Dv)dx,$$
under the constraint $v \ge \psi$ a.e.\ in $\Omega$, where $\psi$ is a fixed obstacle function. Assuming that the coefficients of the partial map $x \mapsto D_\xi F(x,\xi)$ satisfy a suitable Sobolev regularity, we are able to obtain higher differentiability and Lipschitz continuity results for the local minimizers. 
\end{abstract}

\medskip
\noindent \textbf{Keywords}: Regularity; Local minimizer; Higher differentiability; Local Lipschitz continuity.
\medskip \\
\medskip
\noindent \textbf{MSC 2020:} 35J87; 49J40; 47J20.

\section{Introduction}
In this paper, we aim to present some regularity results for the solutions to the problem
\begin{equation}\label{minprob}
\min\left\{ \int_\Omega F(x,Dw) dx \ : \ w\in \mathcal{K}_\psi(\Omega) \cap L^\infty_{loc}(\Omega)
\right\}, 
\end{equation}
where $\Omega$ is a bounded open subset of $\mathbb{R}^n$, for $n\geq 2$, the function $\psi: \Omega \rightarrow [-\infty, + \infty)$, called \textit{obstacle}, belongs to the Sobolev class $W^{1,p}_{loc}(\Omega)$, with $p \ge 2$, and
	$$\mathcal{K}_\psi(\Omega)=\{v\in W^{1, p}_{loc}(\Omega): v\ge\psi \text{ a.e.\ in }\Omega\}$$
	is the class of the admissible functions.
We assume that $F : \Omega \times \mathbb{R}^n \rightarrow [0, + \infty)$ is a Carath\'{e}odory function and there exists a function $ f : \Omega \times [0, + \infty) \rightarrow [0, + \infty)$ satisfying the condition
\begin{equation} \tag{F1}
    F(x, \xi)= f(x, |\xi|) \label{eqf1}
\end{equation}

\noindent for a.e.\ $x \in \Omega$ and every $\xi \in \mathbb{R}^n$.
Moreover, we also assume that there exist positive constants $\nu$, $l$, $L$, $\tilde{L}$, exponents $2 \leq p \le q < + \infty$ such that the following set of assumptions is satisfied:
\begin{equation}\tag{F2}
    l (1+|\xi|^p ) \leq F(x, \xi) \leq L(1+|\xi|^2)^{\frac{q}{2}} \label{eqf2}
\end{equation}
\begin{equation}\tag{F3}
     \langle D_{\xi \xi}F(x, \xi) \lambda, \lambda \rangle \geq \nu (1+ |\xi|^2)^{\frac{p-2}{2}} |\lambda|^2 \label{eqf3} 
\end{equation}
\begin{equation}\tag{F4}
    |D_{\xi \xi }F(x, \xi)| \leq \tilde{L}(1 + |\xi|^2)^{\frac{q-2}{2}} \label{eqf4}
\end{equation}
\begin{equation}\tag{F5}
    |D_{x \xi  }F(x, \xi)| \leq g(x)(1 + |\xi|^2)^{\frac{q-1}{2}} \label{eqf5}
\end{equation}
\noindent for a.e.\ $x,y \in \Omega$ and every $\xi , \lambda \in \mathbb{R}^n$, where $g: \Omega \to \mathbb{R} $ is a non-negative measurable function.

\begin{rmk}
    We notice that assumptions \eqref{eqf3}, \eqref{eqf4} and \eqref{eqf5} imply that there exist positive constants $\tilde{\nu}, \tilde{l}$ and a non-negative measurable function  $k: \Omega \to \mathbb{R} $ such that
    \begin{equation}
	\label{F3*}
	\langle D_\xi F(x,\xi) - D_\xi F(x,\eta), \xi-\eta \rangle\ge \tilde{\nu}(1+|\xi|^{2}+|\eta|^{2})^{\frac{p-2}{2}}|\xi-\eta|^2
	\tag{F3*}
\end{equation}
\begin{equation}
	\label{F4*}
	|D_\xi F(x,\xi) - D_\xi F(x,\eta)| \le \tilde{l}(1+|\xi|^{2}+|\eta|^{2})^{\frac{q-2}{2}}|\xi-\eta|
	\tag{F4*}
\end{equation}
\begin{equation}
	\label{F5*}
	|D_\xi F(x,\xi) - D_\xi F(y,\xi)| \le |x-y| (k(x)+k(y))(1+|\xi|^{2})^{\frac{q-1}{2}}
	\tag{F5*}
\end{equation}
for a.e.\ $x,y \in \Omega$ and every $\xi, \eta \in \mathbb{R}^n$.
\end{rmk}

We will say that a function $F$ satisfies \textit{$(p, q)$-growth conditions} if assumption \eqref{eqf2} is in force. The study of regularity of minima of functionals with non-standard growth of
$(p, q)$-type was initiated by Marcellini in the seminal papers \cite{marcellini1989,marcellini1991}.
When referring to $(p, q)$-growth conditions \eqref{eqf2},
we call the quantity $q/p > 1$ \textit{the gap ratio of the integrand} $F$, or simply, \textit{the gap}.
It is well known that in order to get {some} regularity of the minimizers of functionals with non-standard growth conditions, even the boundedness, a restriction between $p$ and $q$ need to be imposed, usually expressed in the form 
\begin{equation*}
    q \le c(n)p, \qquad \text{with } c(n) \to 1 \ \text{as} \ n \to \infty.
\end{equation*}
We refer to \cite{giaquinta,marcellini1991} for counterexamples.

The study of obstacle problems started with the fundamental works by Stampacchia \cite{stampacchia} and Fichera \cite{fichera} and has since then attracted much attention.
It is usually observed that the regularity of the solutions to the obstacle problems is influenced by the one of the obstacle; for example, for linear obstacle problems, obstacle and solutions have the same regularity \cite{brezis.kinderlehrer,caffarelli.kinderlehrer,kinderlehrer.stampacchia}. This does not apply in the nonlinear setting, hence along the years there have been intense research activities in this direction (see \cite{choe,choe.l,eleuteri.h,michael}, just to mention a few).

In recent years, there has been a considerable interest in analyzing how an extra differentiability of integer or fractional order of the obstacle transfers to the gradient of solutions: for instance we quote \cite{caselli.gentile.giova,eleuteri.passarelli,eleuteri.passarelli1,gentile0,grimaldi0} in the setting of standard growth conditions, \cite{defilippis1,defilippis,gavioli2,gentile1,Gri,grimaldi.ipocoana,grimaldi.ipocoana1,Koch,zhang.zheng} in the setting of non-standard growth conditions.

It is well known that no extra differentiability properties for the solutions can be expected even if the obstacle $\psi$ is smooth, unless some assumption is given on the coefficients of the map $D_\xi F(x,\xi)$.
$W^{1,r}$ Sobolev regularity, with $r \ge n$, or a $B^s_{r,\sigma}$ Besov regularity, with $r \ge \frac{n}{s}$, on the partial map $x \mapsto D_\xi F(x,\xi)$ is a sufficient condition for the higher differentiability of solutions to obstacle problems (see \cite{eleuteri.passarelli,gavioli2} for the case of Sobolev class of integer order and \cite{eleuteri.passarelli,grimaldi.ipocoana} for the fractional one). 

It is worth mentioning that the local boundedness of the solutions to a variational problem is a turning point in the regularity theory. Indeed, when dealing with bounded minimizers, both for unconstrained and constrained problems with $(p,q)$-growth, regularity results for the gradient can be proved under weaker assumptions on the data of the problem and under dimension-free conditions on the gap $q/p$  (see for instance \cite{CPdNK,caselli.gentile.giova,colombo1,dF.M,gentile0,gentile1,Gri,GPdN1}).

In particular, in \cite{caselli.gentile.giova} it has been proved that the higher differentiability of bounded solutions to obstacle problems with standard growth holds true under weaker assumptions on the partial map $x \mapsto D_\xi F(x,\xi)$ with respect to $W^{1,n}$. Moreover, the a priori local boundedness of the solutions allows to prove higher differentiability results assuming that the coefficients of the map $D_\xi F(x,\xi)$ belongs to a Sobolev class that is not related to the dimension $n $ but to the growth exponent of the energy density $F$. For the same phenomenon in the setting of obstacle problems with non-standard growth, see \cite{gentile1}, for the higher differentiability of integer order, and \cite{Gri}, for the higher differentiability of fractional order.
We would like to point out that all the above mentioned results have been obtained assuming that the obstacle is locally bounded. Actually, in \cite{cepdn} it has been proved that local boundedness of the obstacle $\psi$ implies the local boundedness of the solutions to the obstacle problem \eqref{minprob} (for a sharp condition on the gap $q/p$ see \cite{DRG}).

When studying the regularity of the solutions to problems with non-standard growth, one needs to establish \textit{good} a priori estimates and then find a family of approximating problems with standard growth, whose solutions posses the regularity needed to establish the a priori estimates. Therefore,
our first aim is to exploit the analysis of the higher differentiability properties of a priori bounded solutions to \eqref{minprob} under standard growth conditions, which, as far as we know, is not available in literature. The main novelty is that we do not need to assume the local boundedness of the obstacle function.

More precisely, we prove that under standard growth conditions, i.e.\ when $p=q$, the solutions to \eqref{minprob} satisfy some higher differentiability and higher integrability properties.
\begin{thm}\label{highdiff}
    Let $u \in W^{1,p}_{loc}(\Omega) \cap L^{\infty}_{loc}(\Omega)$ be a solution to \eqref{minprob}, under assumptions \eqref{eqf2}, \eqref{F3*} and \eqref{F4*} for exponents $2 \le p =q$. Moreover, assume that \eqref{F5*} holds for a function $k \in L^{p+2}_{loc}(\Omega)$. If $D \psi \in W^{1,p}_{loc}(\Omega)$, then
    $$V_p(Du) \in W^{1,2}_{loc}(\Omega) \quad \text{and} \quad Du \in L^{p+2}_{loc}(\Omega).$$
\end{thm}

The local boundedness of the gradient $Du$ is a main issue for $(p,q)$-growth problems. Indeed, in this case the related functional goes back to the setting of problems with standard growth, since the behaviour of $|Du|$ at infinity becames irrilevant. See, for instance, \cite{beck,bella,DeG} for some Lipschitz continuty results for minimizers
of integral functionals with non-standard growth. {See also the recent paper \cite{GMPdN} in the case of functionals without the radial structure with respect to the gradient variable and for asymptotically convex energy densities.}

A very interesting model case of functional with $(p,q)$-growth is the so-called \textit{double phase functional} defined by
\begin{equation}
    \int_\Omega \left( |Dw|^p+a(x)|Dw|^q \right)dx, \label{DPf}
\end{equation}
with $a(x) \in \mathcal{C}^{0,\alpha}(\Omega)$, that arise in studying the behaviour of strongly anisotropic materials in the context of homogenisation, nonlinear elasticity and
Lavrentiev phenomenon (see, for instance, the papers by Zhikov \cite{Z1,Z2}). The regularity properties of local minimizers to such functionals have been widely investigated in \cite{BA,FM}, with the aim of identifying sufficient and necessary conditions on the relation between $p$, $q$ and $\alpha$ in order to establish the $\mathcal{C}^{1,\beta}$-regularity of minima. The sharp condition on the gap is given by
\begin{equation}\label{GS}
    \dfrac{q}{p} \le 1 + \dfrac{\alpha}{n}.
\end{equation}
In \cite{EMM1}, the authors considered non-autonomous integral functionals with $(p,q)$-growth, without assuming any structure conditions for the integrand as in \eqref{DPf}. They proved that local minimizers are locally Lipschitz continuous assuming a $W^{1,r}$-regularity on the
coefficients, with $r >n$, and if the exponents $p$ and $q$ satisfy the relation
\begin{equation}\label{GAP1}
    \frac{q}{p} < 1+ \dfrac{1}{n}- \dfrac{1}{r} .
\end{equation}
In the model case \eqref{DPf}, the inequality \eqref{GAP1}
gives back \eqref{GS}. Indeed, by the Sobolev embedding theorem, if $a(x) \in W^{1,r}(\Omega)$ then $a(x) \in \mathcal{C}^{0,\alpha}(\Omega)$
with exponent
\begin{equation}\label{alpha}
    \alpha=1-\dfrac{n}{r}.
\end{equation}
On the other hand, in \cite{colombo1} the authors showed that assuming that the minimizers of \eqref{DPf} are a priori bounded yields a gap bound
of the form
\begin{equation}\label{Bgap}
    q \le p+\alpha
\end{equation}
that is sharp as shown in \cite{FM}.

Recently, Eleuteri and Passarelli di Napoli (\cite{EPdN}) investigated the Lipschitz continuity of a priori bounded minimizers of more general functionals with $(p,q)$-growth than the one defined in \eqref{DPf}. {Actually, they are able to prove that the local minimizers are Lipschitz continuous under the following assumption on the gap
\begin{equation}\label{gap}
       q < p+1- \max \left\{ \dfrac{n}{r}, \dfrac{p+2}{r} \right\},
   \end{equation}
which reduces for $p < n-2$ to \eqref{Bgap} by the Sobolev embedding theorem with $\alpha$ as in \eqref{alpha}.
}

Our second aim is to show that, under a suitable Sobolev regularity assumption on the gradient of the obstacle
  and on the partial map $x \mapsto D_\xi F(x,\xi)$, the solutions to \eqref{minprob} are locally Lipschitz continuous.

  A first Lipschitz continuity result can be found in \cite{benassi} in the case of non-autonomous functionals with standard growth conditions.
  The analogous study for solutions to obstacle problems under non-standard
growth has been exploited in \cite{cepdn}, assuming that the gap satisfies the bound \eqref{GAP1}.

 Here, {we show that the regularity result contained in \cite{EPdN} holds true also in the case of obstacle problems.
 Indeed, we prove the following }  
\begin{thm}\label{Mthm}
 Let $u \in W^{1,p}_{loc}(\Omega) \cap L^{\infty}_{loc}(\Omega)$ be a solution to \eqref{minprob}, under assumptions \eqref{eqf1}--\eqref{eqf5}. Let $g \in L^r_{loc}(\Omega)$, with $r >\max \{n, p+2 \}$. Assume that $2 \le p < q$ satisfy \eqref{gap}.
   If $\psi \in W^{2, r}_{loc}(\Omega)$, then $u \in W^{1,\infty}_{loc}(\Omega)$.
\end{thm}

{Differently from the case of standard growth, here the assumption $\psi \in W^{2, r}_{loc}(\Omega)$, that implies that the obstacle function $\psi$ is locally bounded, is needed to prove the Lipschitz continuity of the minimizers. }

We point out that, in Theorems \ref{highdiff} and \ref{Mthm}, we are able to weaken the assumptions on the coefficients of the map $D_\xi F(x,\xi)$ and on gap $q/p$, respectively, with respect to the results established in \cite{cepdn,eleuteri.passarelli}, since the local
boundedness of the solutions allows us to use an interpolation inequality (see Lemma \ref{lemma5_GPdN}) that gives the higher
local integrability $L^{p+2 }$ of the gradient.

Existence of solutions to the obstacle problem \eqref{minprob} can be easily proved through direct methods of the Calculus of Variation, so in this paper we will mainly concentrate on the
regularity results.

The proof of the previous results is based on an approximation procedure, i.e.\ we establish uniform a priori estimates for approximating solutions and
then pass to the limit in the approximating problems. We construct a family of approximating problems, in such a way the approximating local minimizers have norm in a suitable Lebesgue space which is uniformly bounded by the $L^\infty$-norm of $u-\psi$.  This
procedure, aimed at using
the $L^\infty$-information, has been considered for the first time in \cite{CPdNK} in the autonomous case for unconstrained problems; see also \cite{GPdN1} for the non-autonomous case. However, this technique has not yet been exploited in the case of obstacle problems. Here, the main difference with respect to the unconstrained case is the fact that we need to take into account not only the local boundedness of the solution but also that of the obstacle function. This gives rise to a non trivial problem in the proof of Theorem \ref{highdiff}, since the obstacle function is not a priori locally bounded. In order to deal with this issue, we introduce a new obstacle function defined by
$$\tilde{\psi}:= \max \{ \psi, \inf_\Omega u \}.$$
It is clear that $\tilde{\psi}$ is locally bounded, since $u$ is bounded and $u \ge \psi$. The map $\tilde{\psi}$ is a {good} obstacle function, in the sense that a solution to \eqref{minprob} is also a local minimizer of the same integral functional under the different constraint $u \ge \tilde{\psi}$.

 The proof of the Lipschitz continuity consists in several steps. {Here, the main issue is to prove that the approximating solutions satisfy a suitable regularity on the second derivatives.} Once this second order regularity is established, we can argue through a linearization argument, that goes back to \cite{fuchs}. Precisely, we show that the solutions $u_j$ to the approximating problems can be interpreted as the solutions to elliptic equations of the form
\begin{equation}
    \text{div} D_\xi F_j(x,Du_j)=g_j \label{linp}
\end{equation}
with a right hand side, obtained after the identification of a suitable Radon measure. %\textcolor{blue}{We want to point out that we are able to consider the linearized problem \eqref{linp}, since by Theorem \ref{highdiff} the solutions $u_j$ posses a higher differentiability property.  } 

Next, we prove a uniform second order Caccioppoli type inequality
for the approximating minimizers. This, together with a Gagliardo-Nirenberg type interpolation inequality (see Lemma \ref{lemma5_GPdN}), allows us to establish a uniform higher integrability
result for the approximating minimizers, with constants independent of the parameters of the
approximation.  Eventually, we establish a uniform a priori estimate for the $L^\infty$-norm of the gradient of the minimizers of the approximating functionals, then
we show that these estimates are preserved in passing to the limit.

The plan of the paper is briefly described.
In Section \ref{Pre} we recall some notation and preliminary
results.  
In Section \ref{HDsec}, we prove the higher differentiability result (Theorem \ref{highdiff}).
Section \ref{LIP} is devoted to the proof of the Lipschitz regularity result (Theorem \ref{Mthm}).

\section{Preliminary Results}\label{Pre}

In this section we recall some classical definitions and lemmas that we will use in order to prove our results. In what follows, $C$ or $c$ will denote a general constant that may vary on different occasions (even within the same line of estimates). The dependencies on parameters and special constants will be emphasized using either parentheses or subscripts. The norm we use on $\mathbb{R}^n$ will be the standard Euclidean ones and denoted by $|\cdot|$. For $x\in\mathbb{R}^n$ and $r>0$, the symbol $B(x,r)=B_r(x)$ denotes the ball of radius $r$ and center $x$. If the center $x$ is not relevant we will omit it and write only $B_r$. With $(\cdot)_{+}$ we denote the positive part of the argument.\\

\noindent
In order to establish our result the following auxiliary functions will play a fundamental role.
\begin{equation}
	\label{Auxiliary_functions}
	V_p(\xi):=(1+|\xi|^2)^{\frac{p-2}{4}}\xi \quad \text{and} \quad H_p(\xi):=|\xi|^{\frac{p-2}{2}}\xi,
\end{equation}
for all $\xi\in \mathbb{R}^{k}$, $k \ge 1$. 
For the function $V_{p}$, we recall the following estimates (see e.g.\ \cite[Lemma 8.3]{giusti}). 
\begin{lem}\label{D1}
Let $1<p<+\infty$. There exists a constant $c=c(n,p)>0$ such that
\begin{center}
$c^{-1}(1+|\xi|^{2}+|\eta|^{2})^{\frac{p-2}{2}} \leq \dfrac{|V_{p}(\xi)-V_{p}(\eta)|^{2}}{|\xi-\eta|^{2}} \leq c(1+|\xi|^{2}+|\eta|^{2})^{\frac{p-2}{2}} $
\end{center}
for any $\xi, \eta \in \mathbb{R}^{k}$, $\xi \neq \eta$.
\end{lem}

\noindent For a $\mathcal{C}^2$ function $w$, it is easy to check that there exists a positive constant $C(p)$ such that
\begin{equation}\label{IN1}
    C^{-1} |D^2 w |^2 (1+|D w |^2)^\frac{p-2}{2} \le |D (V_p (Dw))|^2 \le C |D^2 w |^2 (1+|D w |^2)^\frac{p-2}{2}
\end{equation}
and
\begin{equation}\label{IN2}
    |D^2 w |^2 |Dw|^{p-2} \le |D(H_p(Dw))|^2 \le \dfrac{p^2}{4} |D^2 w |^2 |Dw|^{p-2} . 
\end{equation}

Now we state a well-known iteration lemma (we refer to \cite[Lemma 6.1]{giusti} for the proof).
\begin{lem}\label{lm2}
Let $\Phi  :  [\frac{R}{2},R] \rightarrow \mathbb{R}$ be a bounded nonnegative function, where $R>0$. Assume that for all $\frac{R}{2} \leq r < s \leq R$ it holds
$$\Phi (r) \leq \theta \Phi(s) +A + \dfrac{B}{(s-r)^2}+ \dfrac{C}{(s-r)^{\gamma}}$$
where $\theta \in (0,1)$, $A$, $B$, $C \geq 0$ and $\gamma >0$ are constants. Then there exists a constant $c=c(\theta, \gamma)$ such that
$$\Phi \biggl(\dfrac{R}{2} \biggr) \leq c \biggl( A+ \dfrac{B}{R^2}+ \dfrac{C}{R^{\gamma}}  \biggr).$$
\end{lem}

The main tools in the proof of Theorem \ref{highdiff} are the following Gagliardo–Nirenberg type inequalities. The proof of inequality \eqref{2.1GP} can be found in \cite[Appendix A]{CPdNK}, while for inequality \eqref{2.2GP} see \cite{Riviere}.
\begin{lem}
	\label{lemma5_GPdN}
		For any $\phi\in C_0^1(\Omega)$ with $\phi\ge0$, and any $C^2$ map $v:\Omega\to\mathbb{R}^N$, $N\ge 1$, we have
		\begin{align}\label{2.1GP}
			\int_\Omega&\phi^{\frac{m}{m+1}(p+2)}(x)|Dv(x)|^{\frac{m}{m+1}(p+2)} dx\notag\\
			\le&(p+2)^2\left(\int_\Omega\phi^{\frac{m}{m+1}(p+2)}(x)|v(x)|^{2m}dx\right)^\frac{1}{m+1}\cdot\left[\left(\int_\Omega\phi^{\frac{m}{m+1}(p+2)}(x)\left|D\phi(x)\right|^2\left|Dv(x)\right|^p dx\right)^\frac{m}{m+1}\right.\notag\\
	&\left.+n\left(\int_\Omega\phi^{\frac{m}{m+1}(p+2)}(x)\left|Dv(x)\right|^{p-2}\left|D^2v(x)\right|^2 dx\right)^\frac{m}{m+1}\right],
		\end{align}
		for any $p\in(1, \infty)$ and $m>1$. Moreover, for any $\mu\in[0,1]$
		\begin{align}\label{2.2GP}
\int_{\Omega}&\phi^2(x)\left(\mu^2+\left|Dv(x)\right|^2\right)^\frac{p}{2}\left|Dv(x)\right|^2 dx\notag\\
			\le&c\Arrowvert v\Arrowvert_{L^\infty\left(\mathrm{supp}(\phi)\right)}^2\int_\Omega\phi^2(x)\left(\mu^2+\left|Dv(x)\right|^2\right)^\frac{p-2}{2}\left|D^2v(x)\right|^2 dx\notag\\
			&+c\Arrowvert v\Arrowvert_{L^\infty\left(\mathrm{supp}(\phi)\right)}^2\int_\Omega\left(\phi^2(x)+\left|D\phi(x)\right|^2\right)\left(\mu^2+\left|Dv(x)\right|^2\right)^\frac{p}{2} dx,
		\end{align}
		for a constant $c=c(p)$.
	\end{lem}

We conclude this section with a technical lemma that will be useful to deal with the {approximating problems},
whose proof can be found in \cite[Lemma 4.1]{Brasco}.
\begin{lem}
	\label{technical_lemma}
	Given $\delta>0$, $m>1$ and $\xi,\eta\in\mathbb{R}^k$, let
	\begin{equation*}
		W(\xi):=(|\xi|-\delta)^{2m-1}_+\frac{\xi}{|\xi|} \quad \text{and}\quad \tilde{W}(\xi):=(|\xi|-\delta)^{2m}_+\frac{\xi}{|\xi|},
	\end{equation*}
then there exists a positive constant $c(m)$ such that
\begin{equation*}
	\langle W(\xi)-W(\eta),\xi-\eta\rangle\ge c(m)|\tilde{W}(\xi)-\tilde{W}(\eta)|^2.
\end{equation*}
\end{lem}

\subsection{Difference Quotients}

\begin{dfn}
	\label{difference_quotients}
	Given a function $f:\R^n\to \R$ the finite difference operator is defined as 
	\begin{equation*}
		\tau_{h}f(x):=f(x+h)-f(x),
	\end{equation*}
with $h\in\R^n$.
\end{dfn}

We start with the description of some elementary properties that can be found, for example, in \cite{giusti}.
\begin{prop}
	\label{difference_quotients_properties}
	Let $f\in
	W^{1,p}(\Omega)$, with $p\geq 1$, and let $g : \Omega \to \mathbb{R}$ be a measurable function. Consider the set
	$$
	\Omega_{|h|}:=\left\{x\in \Omega : \text{dist}(x,
	\partial\Omega)>|h|\right\}, \quad h\in\R^n.
	$$
	Then
	\begin{itemize}
		\item[i)] $\tau_{h}f\in W^{1,p}(\Omega_{|h|})$ and
		$$
		D_{i} (\tau_{h}F)=\tau_{h}(D_{i}F).
		$$
		\item[ii)] If at least one of the functions $f$ or $g$ has support contained
		in $\Omega_{|h|}$, then
		$$
		\int_{\Omega} f(x)\, \tau_{h} g(x)  dx =-\int_{\Omega} g(x)\, \tau_{-h}f(x) 
		dx.
		$$
		\item[iii)] We have
		$$
		\tau_{h}(f g)(x)=f(x+h )\tau_{h}g(x)+g(x)\tau_{h}f(x).
		$$
	\end{itemize}
\end{prop}

The next result about the finite difference operator is a kind of integral version of Lagrange Theorem.
\begin{lem}\label{le1} Let $0<\rho<R$, $|h|<\frac{R-\rho}{2}$, $1 < p <+\infty$ and $f\in W^{1,p}(B_{R})$, then
	$$
	\int_{B_{\rho}} |\tau_{h} f(x)|^{p}\ dx\leq c(n,p)|h|^{p}
	\int_{B_{R}} |D f(x)|^{p} dx .
	$$
	Moreover
	$$
	\int_{B_{\rho}} |f(x+h )|^{p} dx\leq  \int_{B_{R}} |f(x)|^{p} dx .
	$$
\end{lem}

\begin{lem}\label{Giusti8.2}
	Let $f:\R^n\to\R^N$, $N\ge 1$, $f\in L^p(B_R, \R^N)$, and $1<p<+\infty$. Suppose that there exist $\rho\in(0, R)$ and $M>0$ such that
	$$
	\int_{B_\rho}|\tau_{h}f(x)|^p dx\le M^p|h|^p
	$$
	for every $h\in\R^n$ such that $|h|<\frac{R-\rho}{2}$. Then $f\in W^{1,p}(B_R,\R^N)$. Moreover
	$$
	\Arrowvert Df \Arrowvert_{L^p(B_\rho)}\le M.
	$$
\end{lem}

\section{Higher Differentiability}\label{HDsec}

\subsection{Preliminary Higher Regularity Results}\label{PreHD}

In this section we prove a higher differentiability and a higher integrability result for local minimizers $v \in \mathcal{K}_{\tilde{\psi}}(\Omega)\cap L^{2m}_{loc}(\Omega)$ of the integral functional
\begin{equation}\label{IntF}
     \int_{\Omega} \left( F(x,Dw)+(w-\tilde{\psi}-a)_+^{2m} \right) dx,
\end{equation}
with $a >0 $, $m > p/2$ and $\tilde{\psi}= \max \{ \psi, \inf_\Omega u \}$, where $u$ is a fixed local minimizer of \eqref{OP}.

\begin{rmk}
	\label{bounded_from_above}
	We point out that assuming $u\in L^\infty_{loc}(\Omega)$ implies that the obstacle $\psi$ must be bounded from above, since $u\ge\psi$ a.e.\ in $\Omega$. Therefore, by construction we have that the obstacle function $\tilde{\psi}$ is locally bounded in $\Omega$. 
\end{rmk}

The following theorem is fundamental in order to prove the uniform a priori estimates, that are the focus of
Step 2 of the proof of Theorem \ref{highdiff}. The uniform a priori estimates rely on the use of Lemma \ref{lemma5_GPdN} for which it is necessary to have the existence of the second derivatives of the local minimizers as well as the $L^{\frac{m}{m+1}(p+2)}$ integrability of
the gradient.

\begin{thm}
	\label{Preliminary_regularity}
	Let $v\in W^{1,p}_{loc}(\Omega)\cap L^{2m}_{loc}(\Omega)$ be a local minimizer of \eqref{IntF} in $\mathcal{K}_{\tilde{\psi}}(\Omega)$ under the assumptions \eqref{eqf2}, \eqref{F3*} and \eqref{F4*}, for exponents $2 \le p=q$. Moreover, assume that there exists a positive constant $\kappa$ such that for a.e.\ $x,y \in \Omega$ and every $\xi \in \mathbb{R}^n$ 
 \begin{equation}
	\label{F5'}
	|D_\xi F(x,\xi) - D_\xi F(y,\xi)| \le\kappa|x-y|(1+|\xi|^2)^{\frac{p-1}{2}}.
	\tag{F5'}
\end{equation}
If $D \psi \in W^{1,p}_{loc}(\Omega)$, then 
	\begin{equation*}
		(1+|Dv|^2)^{\frac{p-2}{4}}Dv\in W_{loc}^{1,p}(\Omega) \quad \text{and} \quad Dv\in L_{loc}^{\frac{m}{m+1}(p+2)}(\Omega).
	\end{equation*}
\end{thm}

\begin{proof}
We start by observing that, since $F$ satisfies standard $p$-growth conditions, the minimizer $v$ solves the variational inequality
\begin{equation}\label{VI1}
    \int_{\Omega} \langle D_\xi F(x,Dv), D(\varphi - v) \rangle dx + 2m \int_{\Omega} (v-\tilde{\psi}-a)_+^{2m-1}(\varphi-v)dx \ge 0,  
\end{equation}
for all $ \varphi \in \mathcal{K}_{\tilde{\psi}}(\Omega)$.

	Fix $0<R \le 1$ such that $B_R\subset B_{3R}\Subset\Omega$ and fix a cut-off function $\eta\in C_0^\infty(B_{2R})$ such that $\eta=1$ in $B_R$ and $|D\eta|\le \frac{c}{R}$. {We consider 
$$\varphi_1(x)=u(x)+t\eta^2(x)\tau_{h}(v-\tilde{\psi})(x)$$
 and
 $$\varphi_2(x)=u(x)+t\eta^2(x-h)\tau_{-h}(v-\tilde{\psi})(x)$$
 which belong to the admissible class $\mathcal{K}_{\tilde{\psi}}(\Omega)\cap L^{2m}_{loc}(\Omega)$ for all $t \in [0,1)$, since, by Remark \ref{bounded_from_above}, $\tilde{\psi}\in L^\infty_{loc}(\Omega)$.
 Using $\varphi_1$ and $\varphi_2$ in the variational inequality \eqref{VI1}, summing the corresponding inequalities and performing a change of variable, we get}
	\begin{align*}
		&\int_{\Omega'} \langle D_\xi F(x,Dv), \tau_{-h} D(\eta^2\tau_{h}(v-\tilde{\psi})) \rangle dx + 2m\int_{\Omega'} (v-\tilde{\psi}-a)_+^{2m-1} \tau_{-h}(\eta^2\tau_{h}(v-\tilde{\psi}))\dx \ge 0.
	\end{align*}
Hence, by Proposition \ref{difference_quotients_properties}, we infer
\begin{align*}
	&\int_{\Omega} \langle \tau_{h}D_\xi F(x,Dv),  D(\eta^2\tau_{h}(v-\tilde{\psi})) \rangle dx + 2m\int_{\Omega} \tau_{h}\left((v-\tilde{\psi}-a)_+^{2m-1} \right)\eta^2\tau_{h}(v-\tilde{\psi})\dx \le 0.
\end{align*}
%Using the definition of $\tau_{h}$ (see Definition \ref{difference_quotients}) we get
%\begin{align*}
%	&\int_{\Omega} \langle D_\xi f(x+h,Dv(x+h)) - D_\xi f(x,Dv(x)),  D\left[\eta^2\tau_{h}(v-\tilde{\psi})\right] \rangle\\
%	&  + 2m\tau_{h}\left((v-\tilde{\psi}-a)_+^{2m-1} \frac{v-\tilde{\psi}}{|v-\tilde{\psi}|}\right)\eta^2\tau_{h}(v-\tilde{\psi}) \dx \le 0.
%\end{align*}

%Summing and subtracting $D_\xi f(x+h,v(x))$ to the first term of the first line we get

%\begin{align*}
%&\int_{\Omega} \langle D_\xi f(x+h,Dv(x+h)) - D_\xi f(x+h,v(x)),D\left[\eta^2\tau_{h}(v-\tilde{\psi})\right] \rangle\dx\\
%	&+\int_{\Omega}\langle D_\xi f(x+h,v(x)) - D_\xi f(x,Dv(x)),  D\left[\eta^2\tau_{h}(v-\tilde{\psi})\right] \rangle\dx\\
%	&  + 2m\int_{\Omega}\tau_{h}\left((v-\tilde{\psi}-a)_+^{2m-1} \frac{v-\tilde{\psi}}{|v-\tilde{\psi}|}\right)\eta^2\tau_{h}(v-\tilde{\psi}) \dx \le 0.
%\end{align*}
%Since $D\left[\eta^2\tau_{h}(v-\tilde{\psi})\right]=\eta^2\tau_{h}Dv-\eta^2\tau_{h}D\tilde{\psi}+2\eta D\eta \tau_{h}(v-\tilde{\psi})$ we get
\noindent and so
\begin{align*}
	0 \ge &\int_{\Omega} \langle D_\xi F(x+h,Dv(x+h)) - D_\xi F(x+h,Dv(x)),\eta^2\tau_{h}Dv \rangle\dx\\
	&-\int_{\Omega} \langle D_\xi F(x+h,Dv(x+h)) - D_\xi F(x+h,Dv(x)),\eta^2\tau_{h}D\tilde{\psi} \rangle\dx\\
	&+\int_{\Omega} \langle D_\xi F(x+h,Dv(x+h)) - D_\xi F(x+h,Dv(x)),2\eta D\eta \tau_{h}(v-\tilde{\psi}) \rangle\dx\\
	&+\int_{\Omega} \langle D_\xi F(x+h,Dv(x)) - D_\xi F(x,Dv(x)),\eta^2\tau_{h}Dv \rangle\dx\\
	&-\int_{\Omega} \langle D_\xi F(x+h,Dv(x)) - D_\xi F(x,Dv(x)),\eta^2\tau_{h}D\tilde{\psi} \rangle\dx\\
	&+\int_{\Omega} \langle D_\xi F(x+h,Dv(x)) - D_\xi F(x,Dv(x)),2\eta D\eta \tau_{h}(v-\tilde{\psi}) \rangle\dx\\
	&  +2m\int_{\Omega}\tau_{h}\left((v-\tilde{\psi}-a)_+^{2m-1} \right)\eta^2\tau_{h}(v-\tilde{\psi}) \dx\\
	 =: &  I_1+I_2+I_3+I_4+I_5+I_6+I_7.
\end{align*}
It follows that
\begin{equation}
	\label{Stima_delle_I}
	I_1+I_7\le |I_2|+|I_3|+|I_4|+|I_5|+|I_6|.
\end{equation}
From hypothesis \eqref{F3*} we get
\begin{equation*}
	I_1\ge\tilde{\nu}\int_{B_{2R}}\eta^2(1+|Dv(x+h)|^2+|Dv(x)|^2)^{\frac{p-2}{2}}|\tau_{h}Dv|^2\dx.
\end{equation*}
Applying Lemma \ref{technical_lemma} with $\xi=(v-\tilde{\psi})(x+h)$ and $\eta=(v-\tilde{\psi})(x)$, we get
\begin{equation*}
	I_7\ge c(m)\int_{B_{2R}}\eta^2\left|\tau_{h}\left((v-\tilde{\psi})^{2m}_+\right)\right|^2\dx \ge 0,
\end{equation*}
that yields

\begin{equation}
	\label{I_17}
	I_1+I_7\ge\tilde{\nu}\int_{B_{2R}}\eta^2(1+|Dv(x+h)|^2+|Dv(x)|^2)^{\frac{p-2}{2}}|\tau_{h}Dv|^2\dx.
\end{equation}
Now, we take care of the term $I_2$. Using \eqref{F4*}, Young's and H\"older's inequalities and the properties of $\eta$, we have
\begin{align}
	\label{I_2}
	|I_2|\le & \ \tilde{l} \int_{B_{2R}} \eta^2 (1+|Dv(x+h)|^2+|Dv(x)|)^{\frac{p-2}{2}}|\tau_{h}Dv||\tau_{h}D\tilde{\psi}|\dx \nonumber\\
	 \le & \ \e \int_{B_{2R}} \eta^2 (1+|Dv(x+h)|^2+|Dv(x)|)^{\frac{p-2}{2}}|\tau_{h}Dv|^2\dx \nonumber\\
	& + c_\e(\tilde{l}) \int_{B_{2R}} \eta^2 (1+|Dv(x+h)|^2+|Dv(x)|)^{\frac{p-2}{2}}|\tau_{h}D\tilde{\psi}|^2\dx \nonumber\\
	 \le  & \ \e \int_{B_{2R}} \eta^2 (1+|Dv(x+h)|^2+|Dv(x)|)^{\frac{p-2}{2}}|\tau_{h}Dv|^2\dx \nonumber\\
	& + c_\e(\tilde{l}) \left( \int_{B_{3R}} (1+|Dv|)^p\dx \right)^{\frac{p-2}{p}} \left( \int_{B_{2R}} |\tau_{h}D\tilde{\psi}|^p \dx \right)^{\frac{2}{p}} \nonumber\\
	 \le  & \ \e \int_{B_{2R}} \eta^2 (1+|Dv(x+h)|^2+|Dv(x)|)^{\frac{p-2}{2}}|\tau_{h}Dv|^2\dx \nonumber\\
	& + c_\e(\tilde{l})|h|^2 \left( \int_{B_{3R}} (1+|Dv|)^p\dx \right)^{\frac{p-2}{p}} \left( \int_{B_{3R}} |D^2\tilde{\psi}|^p \dx \right)^{\frac{2}{p}}.
\end{align}

\noindent On $I_3$ we use again hypothesis \eqref{F4*} and Young's and H\"older's inequalities, thus getting
\begin{align}
	\label{I_3}
	|I_3|\le & \  2 \tilde{l}\int_{B_{2R}}\eta |D\eta||\tau_{h}(v-\tilde{\psi})|(1+|Dv(x+h)|^2+|Dv(x)|^2)^{\frac{p-2}{2}}|\tau_{h}Dv|\dx \nonumber\\
	\le & \  \e \int_{B_{2R}} \eta^2 (1+|Dv(x+h)|^2+|Dv(x)|)^{\frac{p-2}{2}}|\tau_{h}Dv|^2\dx \nonumber\\
	& + \frac{C_\e(\tilde{l})}{R^2}\int_{B_{2R}}|\tau_{h}(v-\tilde{\psi})|^2(1+|Dv(x+h)|^2+|Dv(x)|)^{\frac{p-2}{2}}\dx\nonumber\\
	 \le & \ \e \int_{B_{2R}} \eta^2 (1+|Dv(x+h)|^2+|Dv(x)|)^{\frac{p-2}{2}}|\tau_{h}Dv|^2\dx \nonumber\\
	& + \frac{C_\e(\tilde{l})}{R^2} \left( \int_{B_{3R}} (1+|Dv|)^p \dx\right)^{\frac{p-2}{p}}\left( \int_{B_{2R}} |\tau_{h}(v-\tilde{\psi})|^p\dx \right)^{\frac{2}{p}}\nonumber\\
	 \le & \  \e \int_{B_{2R}} \eta^2 (1+|Dv(x+h)|^2+|Dv(x)|)^{\frac{p-2}{2}}|\tau_{h}Dv|^2\dx \nonumber\\
	& + \frac{C_\e(n,p,\tilde{l})}{R^2}|h|^2 \left( \int_{B_{3R}} (1+|Dv|)^p \dx\right)^{\frac{p-2}{p}}\left( \int_{B_{3R}} |D(v-\tilde{\psi})|^p\dx \right)^{\frac{2}{p}},
\end{align}
where we also used that $|D \eta | \le \frac{c}{R}$.

In order to estimate the integral $I_4$, we use assumption \eqref{F5'} and next we apply Young's inequality. It follows that
\begin{align}
	\label{I_4}
	|I_4| \le & \kappa \int_{B_{2R}}\eta^2 |h|(1+|Dv^2|)^{\frac{p-1}{2}}|\tau_{h}Dv|\dx\nonumber\\
	\le & \ \e\int_{B_{2R}} \eta^2 (1+|Dv(x+h)|^2+|Dv(x)|^2)^{\frac{p-2}{2}}|\tau_{h}Dv|^2\dx\nonumber\\
	& + C_\e(\kappa)|h|^2 \int_{B_{2R}}(1+|Dv|)^p\dx.
\end{align}
Consider the term $I_5$. Assumption \eqref{F5'}, H\"older's inequality and Lemma \ref{le1} give that
\begin{align}
	\label{I_5}
	|I_5|&\le \kappa \int_{B_{2R}}\eta^2 |h|(1+|Dv^2|)^{\frac{p-1}{2}}|\tau_{h}D\tilde{\psi}|\dx\nonumber\\
	& \le \kappa |h| \left( \int_{B_{2R}} (1+|Dv|)^p\dx \right)^{\frac{p-1}{p}} \left(\int_{B_{2R}} |\tau_{h}D\tilde{\psi}|^p\dx  \right)^{\frac{1}{p}}\nonumber\\
	& \le C(n,p,\kappa) |h|^2 \left( \int_{B_{2R}} (1+|Dv|)^p\dx \right)^{\frac{p-1}{p}} \left(\int_{B_{3R}} |D^2\tilde{\psi}|^p\dx  \right)^{\frac{1}{p}}.
\end{align}
Arguing similarly, the term $I_6$ can be estimated in the following way
\begin{align}
	\label{I_6}
	|I_6|&\le 2\kappa |h|\int_{B_{2R}}\eta|D\eta|(1+|Dv|^2)^{\frac{p-1}{2}}|\tau_{h}(v-\tilde{\psi})|\dx\nonumber\\
	& \le \frac{C(\kappa)}{R}|h|\left( \int_{B_{2R}} (1+|Dv|)^p \dx\right)^{\frac{p-1}{p}}\left( \int_{B_{2R}} |\tau_{h}(v-\tilde{\psi})|^p\dx\right)^{\frac{1}{p}}\nonumber\\
	&\le \frac{C(\kappa)}{R}|h|^2\left( \int_{B_{2R}} (1+|Dv|)^p \dx\right)^{\frac{p-1}{p}}\left( \int_{B_{3R}} |D(v-\tilde{\psi})|^p\dx\right)^{\frac{1}{p}},
\end{align}
where we also used the fact that $|D \eta | \le \frac{c}{R}$.

Plugging \eqref{I_17}, \eqref{I_2}, 
 \eqref{I_3}, \eqref{I_4}, \eqref{I_5}, \eqref{I_6} in \eqref{Stima_delle_I}, we get
\begin{align}
\tilde{\nu} & \int_{B_{2R}} \eta^2(1+|Dv(x+h)|^2+|Dv(x)|^2)^{\frac{p-2}{2}}|\tau_{h}Dv|^2\dx\nonumber\\
	 \le &  \ 3\e \int_{B_{2R}} \eta^2 (1+|Dv(x+h)|^2+|Dv(x)|)^{\frac{p-2}{2}}|\tau_{h}Dv|^2\dx \nonumber\\
	& + C_\e(\tilde{l})|h|^2 \left( \int_{B_{3R}} (1+|Dv|)^p\dx \right)^{\frac{p-2}{p}} \left( \int_{B_{3R}} |D^2\tilde{\psi}|^p \dx \right)^{\frac{2}{p}}\nonumber\\
	& + \frac{C_\e(n,p,\tilde{l})}{R^2}|h|^2 \left( \int_{B_{3R}} (1+|Dv|)^p \dx\right)^{\frac{p-2}{p}}\left( \int_{B_{3R}} |D(v-\tilde{\psi})|^p\dx \right)^{\frac{2}{p}}\nonumber\\
	& + C_\e(\kappa)|h|^2 \int_{B_{2R}}(1+|Dv|)^p\dx\nonumber\\
	&+ C(n,p,\kappa) |h|^2 \left( \int_{B_{2R}} (1+|Dv|)^p\dx \right)^{\frac{p-1}{p}} \left(\int_{B_{3R}} |D^2\tilde{\psi}|^p\dx  \right)^{\frac{1}{p}}\nonumber\\
	& + \frac{C(\kappa)}{R}|h|^2\left( \int_{B_{2R}} (1+|Dv|)^p \dx\right)^{\frac{p-1}{p}}\left( \int_{B_{3R}} |D(v-\tilde{\psi})|^p\dx\right)^{\frac{1}{p}}.
\end{align}

Now, choosing $\e=\frac{\tilde{\nu}}{6}$ and reabsorbing the first term on the right-hand side into the left-hand side, we obtain
\begin{align}
	\label{Penultima_stima}
	\int_{B_{R}} &(1+|Dv(x+h)|^2+|Dv(x)|^2)^{\frac{p-2}{2}}|\tau_{h}Dv|^2\dx\nonumber\\
	\le & \  \frac{c}{R}|h|^2\int_{B_{3R}}(1+|Dv|)^p\dx+c|h|^2\int_{B_{3R}}|D^2\tilde{\psi}|^p\dx\nonumber\\
	&+\frac{c}{R}|h|^2\int_{B_{3R}}|D\tilde{\psi}|^p\dx,
\end{align}
where we also used Young's inequality, the fact that $\eta \equiv 1$ in $B_R$ and $R\le 1$. 

We notice that for fixed $r>0$ such that $B_r\Subset\Omega$, by construction we have that $D \tilde{\psi}= D \psi \ \chi_{ \{ \psi \ge \inf u \}} $ and then it holds
\begin{equation}\label{obstacleapp}
\int_{B_r}|D\tilde{\psi}|^p\dx\le\int_{B_r}|D\psi|^p\dx  \quad \text{and} \quad \int_{B_r}|D^2\tilde{\psi}|^p\dx\le\int_{B_r}|D^2\psi|^p\dx .
\end{equation}
Therefore, from \eqref{Penultima_stima} we get
\begin{align*}
	\int_{B_{R}} &(1+|Dv(x+h)|^2+|Dv(x)|^2)^{\frac{p-2}{2}}|\tau_{h}Dv|^2\dx\nonumber\\
	\le & \ \frac{c}{R}|h|^2\int_{B_{3R}}(1+|Dv|)^p\dx+c|h|^2\int_{B_{3R}}|D^2{\psi}|^p\dx\nonumber\\
	&+\frac{c}{R}|h|^2\int_{B_{3R}}|D{\psi}|^p\dx.
\end{align*}
Now, Lemmas \ref{D1} and \ref{Giusti8.2} yield that
\begin{equation*}
	\int_{B_{R}}|DV_p(Dv)|^2\dx\le {c}\int_{B_{3R}}(1+|Dv|)^p\dx+c\int_{B_{3R}}|D^2{\psi}|^p\dx+{c}\int_{B_{3R}}|D{\psi}|^p\dx,
\end{equation*}
for some constant $c=c(n,p,\tilde{\nu}, \tilde{l}, \kappa, R)$. Moreover, from the left inequality in \eqref{IN1} we derive
\begin{equation*}
	\int_{B_{R}}(1+|Dv|^2)^{\frac{p-2}{2}}|D^2v|^2\dx\le {c}\int_{B_{3R}}(1+|Dv|)^p\dx+c\int_{B_{3R}}|D^2{\psi}|^p\dx+{c}\int_{B_{3R}}|D{\psi}|^p\dx,
\end{equation*}
and
\begin{equation}\label{SecDer}
 	\int_{B_{R}}|D^2v|^2\dx\le {c}\int_{B_{3R}}(1+|Dv|)^p\dx+c\int_{B_{3R}}|D^2{\psi}|^p\dx+{c}\int_{B_{3R}}|D{\psi}|^p\dx.
\end{equation}
Eventually, the interpolation inequality \eqref{2.1GP} in Lemma \ref{lemma5_GPdN} implies that $Dv\in L^{\frac{m}{m+1}(p+2)}_{loc}(\Omega)$.
\end{proof}

\subsection{Proof of Theorem \ref{highdiff}}\label{ProofHD}
In this section we prove Theorem \ref{highdiff}. The proof consists of three steps. First, we construct a sequence of approximating problems using convolution (see \cite{CPdNK}). Then, we prove uniform higher differentiability for the solutions to the approximating problems. Finally, we show that the aforementioned solutions converge to the solution to the original obstacle problem and that the higher differentiability is maintained through the limit. 
\\

\noindent
{\bf Step 1.} Fix $r>0$ such that $B_r \Subset \Omega$. Let $\phi\in C_0^\infty(B_1(0))$ such that $\int_{B_1(0)}\phi \dx=1$ and let $(\phi_\e)_{\e>0}$ be the related family of mollifiers. Fixed $\e<\dist(B_r,\Omega)$, for every $x\in\Omega$ we define
\begin{equation*}
	k_\e(x):=k*\phi_\e(x)
\end{equation*}
and for every $\xi\in\R^n$
\begin{equation*}
F_\e(x,\xi):=\int_{B_1(0)}\phi(y)F(x+\e y,\xi) dy.
\end{equation*}
By construction, $F_\varepsilon$ satisfies \eqref{eqf1}, \eqref{eqf2}, \eqref{F3*}, \eqref{F4*} and the assumption
\begin{equation}
	\label{F5''}
	|D_\xi F_\e(x,\xi) - D_\xi F_\e(y,\xi)| \le |x-y| (k_\e(x)+k_\e(y))(1+|\xi|^2)^{\frac{p-1}{2}}
	\tag{F5''}
\end{equation}
for every $\xi\in\R^n$ and for a.e.\ $x,y\in\Omega$. Now, for fixed $m > p/2$ and $a\ge 2 ||u||_{L^\infty(B_r)}$, we consider the variational problem
\begin{equation}\label{AOP}
	\min \left\{ \int_{B_r} \left( F_\e(x,Dv)+(v-\tilde{\psi}-a)_+^{2m} \right)\dx \ : \ v \in \mathcal{K}_{\tilde{\psi}}(B_r)\cap L^{2m}_{loc}(B_r), \ v=u \text{ on } \partial B_r \right\}.
\end{equation}
Let $u_{\e,m} $ be the local minimizer of \eqref{AOP}. By the minimality of $u_{\e,m}$, it holds
\begin{equation*}
	\int_{B_r} \left(F_\e(x,Du_{\e,m})+(u_{\e,m}-\tilde{\psi}-a)_+^{2m} \right)\dx\le \int_{B_r} F_\e(x,Du)\dx,
\end{equation*}
thus from \eqref{eqf2} it follows that
\begin{align}
	\label{disug1}
	\int_{B_r}|Du_{\e,m}|^p\dx&\le C(l) \int_{B_r} \left(F_\e(x,Du_{\e,m})+(u_{\e,m}-\tilde{\psi}-a)_+^{2m} \right)\dx\nonumber\\
	&\le C(l)\int_{B_r} F_\e(x,Du)\dx\le C(l,L) \int_{B_r} (1+|Du|^p)\dx.
\end{align}
The previous inequality implies that the sequence $(Du_{\e,m})$ is bounded in $L^p(B_r)$, so there exists $u_\e\in W^{1,p}(B_r)$ such that 
\begin{equation}
    u_{\e,m}\rightharpoonup u_\e \quad \text{weakly in }  W^{1,p}(B_r) \text{ as }  m\to \infty. \label{WeakConv}
\end{equation}
Furthermore, by weak lower semicontinuity, we have
\begin{equation*}
\int_{B_r}|Du_\e|^p\dx\le \liminf_m \int_{B_r} |Du_{\varepsilon,m}|^p dx \le C(l,L)\int_{B_r}(1+|Du|)^p\dx,
\end{equation*}
and so, again, $(u_\e)_{\e>0}$ is bounded in $W^{1,p}(B_r)$ and there exists $v\in W^{1,p}(B_r)$ such that
\begin{equation*}
	u_\e\rightharpoonup v \quad \text{in }W^{1,p}(B_r) \text{ as } \e \to 0.
\end{equation*}
For further needs, we record that from \eqref{disug1} we have
\begin{align}
	\int_{B_r}(u_{\e,m}-\tilde{\psi}-a)_+^{2m}\dx\le & \int_{B_r} \left(F_\e(x,Du_{\e,m})+(u_{\e,m}-\tilde{\psi}-a)_+^{2m} \right)\dx \notag\\
 \le & \ C(l,L)\int_{B_r}(1+|Du|^p)\dx. \label{psiS}
\end{align}
Finally, we notice that for every $B_R \subset B_r$ it holds
\begin{equation}\label{uemS}
    \int_{B_R}|u_{\e,m}|^{2m} dx \le c 2^{2m} \left(  a^{2m}R^n + \int_{B_r}(1+|Du|)^p dx\right),
\end{equation}
for a constant $c$ independent of $m$ and $\varepsilon$. Indeed, we have
\begin{align*}
     & \int_{B_R}|u_{\e,m}|^{2m} dx \notag\\
      &= \int_{\{ u_{\e,m}-\tilde{\psi}  \ \le \  a\} \cap B_R}|u_{\e,m}|^{2m}  dx + \int_{\{ u_{\e,m}-\tilde{\psi}  \ > \  a\} \cap B_R}|u_{\e,m}|^{2m}  dx \\
      &= \int_{\{ u_{\e,m}-\tilde{\psi}  \ \le \  a\} \cap B_R}|u_{\e,m}- \tilde{\psi}+\tilde{\psi}|^{2m}  dx + \int_{\{ u_{\e,m}-\tilde{\psi}  \ > \  a\} \cap B_R}|u_{\e,m}-\tilde{\psi}-a+\tilde{\psi}+a|^{2m} dx \\
      & \le  2^{2m}\int_{\{ u_{\e,m}-\tilde{\psi}  \ \le \  a\} \cap B_R}|u_{\e,m}- \tilde{\psi}|^{2m}  dx+ 2^{2m}\int_{\{ u_{\e,m}-\tilde{\psi}  \ \le \  a\} \cap B_R}|\tilde{\psi}|^{2m}  dx \notag\\
      & \ \ \ \ +  2^{2m}\int_{\{ u_{\e,m}-\tilde{\psi}  \ > \  a\} \cap B_R}|u_{\e,m}- \tilde{\psi}-a|^{2m}  dx+ 2^{2m}\int_{\{ u_{\e,m}-\tilde{\psi}  \ > \  a\} \cap B_R}|\tilde{\psi}+a|^{2m}  dx.
\end{align*}
Using the fact that $a\ge 2 ||u||_{L^\infty(B_r)} \ge  ||u||_{L^\infty(B_r)} + ||\tilde{\psi}||_{L^\infty(B_r)} $ and estimates \eqref{psiS} and \eqref{disug1}, we get
\begin{align*}
     \int_{B_R}|u_{\e,m}|^{2m} dx 
    &\le c 2^{2m} \left( a^{2m} R^n + \int_{ B_r}(u_{\e,m}- \tilde{\psi}-a)_+^{2m}  dx  \right) \\
    & \le c 2^{2m} \left( a^{2m} R^n + \int_{B_r} (1+|Du_{\e,m}|)^pdx  \right) \\
     & \le c 2^{2m} \left( a^{2m} R^n + \int_{B_r} (1+|Du|)^pdx  \right).
\end{align*}
\\

\noindent {\bf Step 2.} Fix $R>0$ such that $B_{3R}\subset B_r$ and let $\eta\in C_0^\infty(B_{2R})$ be a cut-off function such that $\eta=1$ on $B_R$ and $|D\eta|\le \frac{C}{R}$.
We consider 
$$\varphi_1(x)=u(x)+t\eta^{p+2}(x)\tau_{h}(v-\tilde{\psi})(x)$$
 and
 $$\varphi_2(x)=u(x)+t\eta^{p+2}(x-h)\tau_{-h}(v-\tilde{\psi})(x)$$
 which belong to the admissible class $\mathcal{K}_{\tilde{\psi}}(\Omega)\cap L^{2m}_{loc}(\Omega)$ for all $t \in [0,1)$, since, by Remark \ref{bounded_from_above}, $\tilde{\psi}\in L^\infty_{loc}(\Omega)$.
 Using $\varphi_1$ and $\varphi_2$ in the variational inequality of the problem \eqref{AOP}, summing the corresponding inequalities and performing a change of variable, we obtain 
 \begin{align*}
 	\int_{B_r}& \langle D_\xi F_\e(x,Du_{\e,m}), \tau_{-h}D(\eta^{p+2}\tau_{h}(u_{\e,m}-\tilde{\psi})) \rangle dx\\
 	&+ 2m \int_{B_r}(u_{\e,m}-\tilde{\psi}-a)_+^{2m-1}\tau_{-h}(\eta^{p+2}\tau_{h}(u_{\e,m}-\tilde{\psi}))\dx\ge 0.
 \end{align*}
 From Proposition \ref{difference_quotients_properties}, we have
 \begin{align*}
 	\int_{B_r} & \langle \tau_{h}D_\xi F_\e(x,Du_{\e,m}), D(\eta^{p+2}\tau_{h}(u_{\e,m}-\tilde{\psi})) \rangle  dx\\
 	&+ 2m \int_{B_r}\tau_{h}\left[(u_{\e,m}-\tilde{\psi}-a)_+^{2m-1}\right]\eta^{p+2}\tau_{h}(u_{\e,m}-\tilde{\psi})\dx\le 0,
 \end{align*}
and so
\begin{align*}
	0 \ge &\int_{B_r}\eta^{p+2}\langle D_\xi F_\e(x+h,Du_{\e,m}(x+h))-D_\xi F_\e(x+h,Du_{\e,m}(x)),\tau_{h}Du_{\e,m}  \rangle\dx\\
	&-\int_{B_r}\eta^{p+2}\langle D_\xi F_\e(x+h,Du_{\e,m}(x+h))-D_\xi F_\e(x+h,Du_{\e,m}(x)),\tau_{h}D\tilde{\psi}  \rangle\dx\\
	&+(p+2)\int_{B_r}\eta^{p+1} \tau_{h}(u_{\e,m}-\tilde{\psi}) \langle D_\xi F_\e(x+h,Du_{\e,m}(x+h))-D_\xi F_\e(x+h,Du_{\e,m}(x)),D\eta \rangle\dx\\
	&+ \int_{B_r}\eta^{p+2}\langle D_\xi F_\e(x+h,Du_{\e,m}(x))-D_\xi F_\e(x,Du_{\e,m}(x)),\tau_{h}Du_{\e,m}  \rangle\dx\\
	&-\int_{B_r}\eta^{p+2}\langle D_\xi F_\e(x+h,Du_{\e,m}(x))-D_\xi F_\e(x,Du_{\e,m}(x)),\tau_{h}D\tilde{\psi}  \rangle\dx\\
	&+(p+2)\int_{B_r}\eta^{p+1} \tau_{h}(u_{\e,m}-\tilde{\psi}) \langle D_\xi F_\e(x+h,Du_{\e,m}(x))-D_\xi F_\e(x,Du_{\e,m}(x)),D\eta  \rangle\dx\\
	&+2m\int_{B_r}\tau_{h}\left[(u_{\e,m}-\tilde{\psi}-a)_+^{2m-1}\right]\eta^{p+2}\tau_{h}(u_{\e,m}-\tilde{\psi})\dx\\
	=: & I_1+I_2+I_3+I_4+I_5+I_6+I_7
\end{align*}
that implies
\begin{equation}
\label{stima_delle_I_epsilon}
	I_1+I_7\le |I_2|+|I_3|+|I_4|+|I_5|+|I_6|.
\end{equation}
By hypothesis \eqref{F3*}, we get
\begin{equation*}
	I_1\ge\tilde{\nu}\int_{B_{2R}}\eta^{p+2}(1+|Du_{\e,m}(x+h)|^2+|Du_{\e,m}(x)|^2)^{\frac{p-2}{2}}|\tau_{h}Du_{\e,m}|^2\dx,
\end{equation*} 
while, from Lemma \ref{technical_lemma}, we get
\begin{equation*}
	I_7\ge c(m)\int_{B_{2R}}\eta^{p+2}\left| \tau_{h}(u_{\e,m}-\tilde{\psi}-a)_+^{2m} \right|^2\dx\ge 0,
\end{equation*}
thus
\begin{equation}
	\label{I_1+I_7.}
	I_1+I_7\ge \tilde{\nu}\int_{B_{2R}}\eta^{p+2}(1+|Du_{\e,m}(x+h)|^2+|Du_{\e,m}(x)|^2)^{\frac{p-2}{2}}|\tau_{h}Du_{\e,m}|^2\dx.
\end{equation}
By virtue of assumption \eqref{F4*}, Young's and H\"older's inequalities and the properties of $\eta$, we have
\begin{align}
	\label{I_2.}
	|I_2|\le & \ \tilde{l} \int_{B_{2R}}\eta^{p+2}(1+|Du_{\e,m}(x+h)|^2+|Du_{\e,m}(x)|^2)^{\frac{p-2}{2}}|\tau_{h}Du_{\e,m}||\tau_{h}D\tilde{\psi}|\dx\nonumber\\
	\le & \ \sigma \int_{B_{2R}}\eta^{p+2}(1+|Du_{\e,m}(x+h)|^2+|Du_{\e,m}(x)|^2)^{\frac{p-2}{2}}|\tau_{h}Du_{\e,m}|^2\dx\nonumber\\
	&+  C_\sigma(\tilde{l})\int_{B_{2R}}\eta^{p+2}(1+|Du_{\e,m}(x+h)|^2+|Du_{\e,m}(x)|^2)^{\frac{p-2}{2}}|\tau_{h}D\tilde{\psi}|^2\dx\nonumber\\
	\le & \ \sigma \int_{B_{2R}}\eta^{p+2}(1+|Du_{\e,m}(x+h)|^2+|Du_{\e,m}(x)|^2)^{\frac{p-2}{2}}|\tau_{h}Du_{\e,m}|^2\dx\nonumber\\
	&+ C_\sigma(\tilde{l})\left(\int_{B_{2R}} 1+|Du_{\e,m}|^p\dx\right)^{\frac{p-2}{p}}\left(\int_{B_{2R}}|\tau_{h}D\tilde{\psi}|^p\dx\right)^{\frac{2}{p}}\nonumber\\
	\le & \  \sigma \int_{B_{2R}}\eta^{p+2}(1+|Du_{\e,m}(x+h)|^2+|Du_{\e,m}(x)|^2)^{\frac{p-2}{2}}|\tau_{h}Du_{\e,m}|^2\dx\nonumber\\
	&+C_\sigma (\tilde{l}) |h|^2 \left(\int_{B_{2R}} (1+|Du_{\e,m}|^p)\dx\right)^{\frac{p-2}{p}}\left(\int_{B_{3R}}|D^2\tilde{\psi}|^p\dx\right)^{\frac{2}{p}}.
\end{align}
Now, we take care of the integral $I_3$. Using again hypothesis \eqref{F4*}, Young's and H\"older's inequalities and the fact that $|D \eta | \le \frac{C}{R}$, we derive
\begin{align}
	\label{I_3.}
	|I_3|\le& \ \tilde{l}(p+2)\int_{B_{2R}}\eta^{p+1}|D\eta||\tau_{h}(u_{\e,m}-\tilde{\psi})|
	(1+|Du_{\e,m}(x+h)|^2+|Du_{\e,m}(x)^2|)^{\frac{p-2}{2}}|\tau_{h}Du_{\e,m}|\dx\nonumber\\
	\le& \ \sigma \int_{B_{2R}}(1+|Du_{\e,m}(x+h)|^2+|Du_{\e,m}(x)^2|)^{\frac{p-2}{2}}|\tau_{h}Du_{\e,m}|^2\dx\nonumber\\
	&+\frac{C_\sigma(\tilde{l},n,p)}{R^2} \int_{B_{2R}}|\tau_{h}(u_{\e,m}-\tilde{\psi})|^2(1+|Du_{\e,m}(x+h)|^2+|Du_{\e,m}(x)^2|)^{\frac{p-2}{2}}\dx\nonumber\\
	\le& \ \sigma\int_{B_{2R}}(1+|Du_{\e,m}(x+h)|^2+|Du_{\e,m}(x)^2|)^{\frac{p-2}{2}}|\tau_{h}Du_{\e,m}|^2\dx\nonumber\\
	& +\frac{C_\sigma(\tilde{l},n,p)}{R^2}\left( \int_{B_{3R}}1+ |Du_{\e,m}|^p\dx \right)^{\frac{p-2}{p}}\left( \int_{B_{2R}} |\tau_{h}(u_{\e,m}-\tilde{\psi})|^p\dx\right)^{\frac{2}{p}}\nonumber\\
		\le& \ \sigma\int_{B_{2R}}(1+|Du_{\e,m}(x+h)|^2+|Du_{\e,m}(x)^2|)^{\frac{p-2}{2}}|\tau_{h}Du_{\e,m}|^2\dx\nonumber\\
	& +\frac{C_\sigma(\tilde{l},n,p)}{R^2} |h|^2\left( \int_{B_{3R}}(1+ |Du_{\e,m}|^p)\dx \right)^{\frac{p-2}{p}}\left( \int_{B_{3R}} |D(u_{\e,m}-\tilde{\psi})|^p\dx\right)^{\frac{2}{p}}.
\end{align}
By assumption \eqref{F5''} and Young's inequality, we can estimate the term $I_4$ as follows
\begin{align}
	\label{I_4.}
	|I_4|\le&   \int_{B_{2R}}\eta^{p+2}(k_\e(x+h)+k_\e(x))|h|(1+|Du_{\e,m}|^2)^{\frac{p-1}{2}}|\tau_{h}Du_{\e,m}|\dx\nonumber\\
	\le& \ \sigma\int_{B_{2R}}(1+|Du_{\e,m}(x+h)|^2+|Du_{\e,m}(x)^2|)^{\frac{p-2}{2}}|\tau_{h}Du_{\e,m}|^2\dx\nonumber\\
 &+C_\sigma|h|^2\int_{B_{2R}}\eta^{p+2}(k_\e(x+h)+k_\e(x))^2(1+|Du_{\e,m}|^2)^{\frac{p}{2}}\dx.
\end{align}
In order to estimate the integral $I_5$, we use assumption \eqref{F5''}, H\"older's inequality and Lemma \ref{le1}, thus getting 
\begin{align}
	\label{I_5.}
	|I_5|\le & \int_{B_{2R}}\eta^{p+2}(k_\e(x+h)+k_\e(x))|h|(1+|Du_{\e,m}|^2)^{\frac{p-1}{2}}|\tau_{h}D\tilde{\psi}|\dx\nonumber\\
	\le & \int_{B_{2R}}(1+|Du_{\e,m}(x)^2|)^{\frac{p-2}{2}}|\tau_{h}D\tilde{\psi}|^2\dx\nonumber\\
&+C(p)|h|^2\int_{B_{2R}}\eta^{p+2}(k_\e(x+h)+k_\e(x))^2(1+|Du_{\e,m}|^2)^{\frac{p}{2}}\dx\nonumber\\
\le & \left( \int_{B_{2R}}(1+|Du_{\e,m}(x)^2|)^{\frac{p}{2}}\dx \right)^{\frac{p-2}{p}}  \left(\int_{B_{2R}}|\tau_{h}D\tilde{\psi}|^p\dx \right)^{\frac{2}{p}}\nonumber\\
&+C(p)|h|^2\int_{B_{2R}}\eta^{p+2}(k_\e(x+h)+k_\e(x))^2(1+|Du_{\e,m}|^2)^{\frac{p}{2}}\dx\nonumber\\
\le& \ C(n,p) |h|^2\left( \int_{B_{2R}}(1+|Du_{\e,m}(x)^2|)^{\frac{p}{2}}\dx \right)^{\frac{p-2}{p}}  \left(\int_{B_{3R}}|D^2\tilde{\psi}|^p\dx \right)^{\frac{2}{p}}\nonumber\\
&+C(p)|h|^2\int_{B_{2R}}\eta^{p+2}(k_\e(x+h)+k_\e(x))^2(1+|Du_{\e,m}|^2)^{\frac{p}{2}}\dx.
\end{align}
Arguing analogously, we get the following estimate for $I_6$.
\begin{align}
	\label{I_6.}
	|I_6|\le& (p+2)\int_{B_{2R}}\eta^{p+1}|D\eta|(k_\e(x+h)+k_\e(x))|h|(1+|Du_{\e,m}|^2)^{\frac{p-1}{2}}|\tau_{h}(u_{\e,m}-\tilde{\psi})|\dx\nonumber\\
	\le & \ C(p)|h|^2\int_{B_{2R}}\eta^{p+2}(k_\e(x+h)+k_\e(x))^2(1+|Du_{\e,m}|^2)^{\frac{p}{2}}\dx\nonumber\\
	&+\frac{C}{R^2}\int_{B_{2R}}\eta^p(1+|Du_{\e,m}|^2)^{\frac{p-2}{2}}|\tau_{h}(u_{\e,m}-\tilde{\psi})|^2\dx\nonumber\\
	\le& \ C(p)|h|^2\int_{B_{2R}}\eta^{p+2}(k_\e(x+h)+k_\e(x))^2(1+|Du_{\e,m}|^2)^{\frac{p}{2}}\dx\nonumber\\
	&+\frac{C}{R^2}\left(\int_{B_{2R}} (1+|Du_{\e,m}|^2)^{\frac{p}{2}} dx\right)^{\frac{p-2}{2}} \left(\int_{B_{2R}} |\tau_{h}(u_{\e,m}-\tilde{\psi})|^p\dx \right)^{\frac{2}{p}}\nonumber\\
	\le& \ C(p)|h|^2\int_{B_{2R}}\eta^{p+2}(k_\e(x+h)+k_\e(x))^2(1+|Du_{\e,m}|^2)^{\frac{p}{2}}\dx\nonumber\\
	&+\frac{C}{R^2}|h|^2\left(\int_{B_{2R}} (1+|Du_{\e,m}|^2)^{\frac{p}{2}} dx\right)^{\frac{p-2}{2}} \left(\int_{B_{3R}} |D(u_{\e,m}-\tilde{\psi})|^p\dx \right)^{\frac{2}{p}}.
\end{align}
Plugging \eqref{I_1+I_7.}, \eqref{I_2.}, \eqref{I_3.}, \eqref{I_4.}, \eqref{I_5.} and \eqref{I_6.} in \eqref{stima_delle_I_epsilon} we get
\begin{align}
    \tilde{\nu}&\int_{B_{2R}}\eta^{p+2}(1+|Du_{\e,m}(x+h)|^2+|Du_{\e,m}(x)|^2)^{\frac{p-2}{2}}|\tau_{h}Du_{\e,m}|^2\dx\nonumber\\
    \le & \ 3 \sigma\int_{B_{2R}}\eta^{p+2}(1+|Du_{\e,m}(x+h)|^2+|Du_{\e,m}(x)|^2)^{\frac{p-2}{2}}|\tau_{h}Du_{\e,m}|^2\dx\nonumber\\
	&+C_\sigma(\tilde{l})|h|^2\left(\int_{2R} (1+|Du_{\e,m}|^p)\dx\right)^{\frac{p-2}{p}}\left(\int_{B_{3R}}|D^2\tilde{\psi}|^p\dx\right)^{\frac{2}{p}}\nonumber\\
 & +\frac{C_\sigma(\tilde{l},n,p)}{R^2}|h|^2\left( \int_{B_{3R}}(1+ |Du_{\e,m}|^p)\dx \right)^{\frac{p-2}{p}}\left( \int_{B_{3R}} |D(u_{\e,m}-\tilde{\psi})|^p\dx\right)^{\frac{2}{p}}\nonumber\\
 &+C_\sigma|h|^2\int_{B_{2R}}\eta^{p+2}(k_\e(x+h)+k_\e(x))^2(1+|Du_{\e,m}|^2)^{\frac{p}{2}}\dx\nonumber\\
 &+ C(n,p)|h|^2\left( \int_{B_{2R}}(1+|Du_{\e,m}(x)^2|)^{\frac{p}{2}}\dx \right)^{\frac{p-2}{p}}  \left(\int_{B_{3R}}|D^2\tilde{\psi}|^p\dx \right)^{\frac{2}{p}}\nonumber\\
&+C(p)|h|^2\int_{B_{2R}}\eta^{p+2}(k_\e(x+h)+k_\e(x))^2(1+|Du_{\e,m}|^2)^{\frac{p}{2}}\dx\nonumber\\
	&+\frac{C}{R^2}|h|^2\left(\int_{B_{2R}} (1+|Du_{\e,m}|^2)^{\frac{p}{2}} dx\right)^{\frac{p-2}{2}} \left(\int_{B_{3R}} |D(u_{\e,m}-\tilde{\psi})|^p\dx \right)^{\frac{2}{p}}. \notag
\end{align}
Choosing $\sigma =\frac{\tilde{\nu}}{6}$, we can reabsorb the first integral in the right-hand side into the left-hand side and next applying Young's inequality, we find that
\begin{align}
    &\int_{B_{2R}}\eta^{p+2}(1+|Du_{\e,m}(x+h)|^2+|Du_{\e,m}(x)|^2)^{\frac{p-2}{2}}|\tau_{h}Du_{\e,m}|^2\dx\nonumber\\
    & \le  \dfrac{C}{R^2}|h|^2 \int_{3R} (1+|Du_{\e,m}|^p)\dx
 +\frac{C}{R^2}|h|^2 \int_{B_{3R}} |D\tilde{\psi}|^p dx + C|h|^2 \int_{B_{3R}} |D^2 \tilde{\psi}|^p dx\nonumber\\
 & \ \ \ \ +C|h|^2\int_{B_{2R}}\eta^{p+2}(k_\e(x+h)+k_\e(x))^2(1+|Du_{\e,m}|^2)^{\frac{p}{2}}\dx,\nonumber
\end{align}
for a positive constant $C:=C(n,p,\tilde{l},\tilde{\nu})$ independent of $m$ and $\varepsilon$. Using Lemma \ref{D1} in the left-hand side of the previous estimate, we obtain
\begin{align}
    &\int_{B_{2R}}\eta^{p+2}|\tau_h V_p(Du_{\e,m})|^2\dx\nonumber\\
    & \le  \dfrac{C}{R^2}|h|^2 \int_{3R} (1+|Du_{\e,m}|^p)\dx
 +\frac{C}{R^2}|h|^2 \int_{B_{3R}} |D{\psi}|^p dx + C|h|^2 \int_{B_{3R}} |D^2 {\psi}|^p dx\nonumber\\
 & \ \ \ \ +C|h|^2\int_{B_{2R}}\eta^{p+2}(k_\e(x+h)+k_\e(x))^2(1+|Du_{\e,m}|^2)^{\frac{p}{2}}\dx,\nonumber
\end{align}
where in right-hand side we used \eqref{obstacleapp}. Hence, by property (iii) of Proposition \ref{difference_quotients_properties}, we get
\begin{align}
    &\int_{B_{2R}}|\tau_h(\eta^\frac{p+2}{2} V_p(Du_{\e,m}))|^2\dx\nonumber\\
    &\le C\int_{B_{2R}}|\tau_h \eta^{p+2}|^2 |V_p(Du_{\e,m}))|^2\dx+  C\int_{B_{2R}}\eta^{p+2} |\tau_h V_p(Du_{\e,m}))|^2\dx \nonumber\\
    & \le  \dfrac{C}{R^2}|h|^2 \int_{3R} (1+|Du_{\e,m}|^p)\dx
 +\frac{C}{R^2}|h|^2 \int_{B_{3R}} |D{\psi}|^p dx + C|h|^2 \int_{B_{3R}} |D^2 {\psi}|^p dx\nonumber\\
 & \ \ \ \ + C \int_{3R} |D \eta|^2 (1+|Du_{\e, m}|)^pdx \notag\\
 & \ \ \ \ +C|h|^2\int_{B_{2R}}\eta^{p+2}(k_\e(x+h)+k_\e(x))^2(1+|Du_{\e,m}|^2)^{\frac{p}{2}}\dx\nonumber \\
 & \le  \dfrac{C}{R^2}|h|^2 \int_{3R} (1+|Du_{\e,m}|^p)\dx
 +\frac{C}{R^2}|h|^2 \int_{B_{3R}} |D{\psi}|^p dx + C|h|^2 \int_{B_{3R}} |D^2 {\psi}|^p dx\nonumber\\
 & \ \ \ \ +C|h|^2\int_{B_{2R}}\eta^{p+2}(k_\e(x+h)+k_\e(x))^2(1+|Du_{\e,m}|^2)^{\frac{p}{2}}\dx.\nonumber
\end{align}

By virtue of Theorem \ref{Preliminary_regularity}, we have
$$Du_{\e,m}\in L_{loc}^{\frac{m}{m+1}(p+2)}(\Omega).$$
Therefore, since $m > p/2$, we can apply H\"older's inequality with exponents $\frac{m}{m+1}\frac{p+2}{p}$ and $\frac{m(p+2)}{2m-p}$, thus getting
\begin{align}
    &\int_{B_{2R}}|\tau_h(\eta^\frac{p+2}{2} V_p(Du_{\e,m}))|^2\dx\nonumber\\
    & \le  \dfrac{C}{R^2}|h|^2 \int_{3R} (1+|Du_{\e,m}|^p)\dx
 +\frac{C}{R^2}|h|^2 \int_{B_{3R}} |D{\psi}|^p dx + C|h|^2 \int_{B_{3R}} |D^2 {\psi}|^p dx\nonumber\\
 & \ \ \ \ +C|h|^2 \left(\int_{B_{3R}}k_\e(x)^\frac{2m(p+2)}{2m-p} dx \right)^\frac{2m-p}{m(p+2)}\left(\int_{B_{2R}}\eta^{p+2}(1+|Du_{\e,m}|)^{\frac{m}{m+1}(p+2)}\dx \right)^{\frac{m+1}{m} \frac{p}{p+2}}.\nonumber
\end{align}
From Lemma \ref{le1}, we deduce that
\begin{align}
    &\int_{B_{2R}}|D(\eta^\frac{p+2}{2} V_p(Du_{\e,m}))|^2\dx\nonumber\\
    & \le  \dfrac{C}{R^2} \int_{3R} (1+|Du_{\e,m}|^p)\dx
 +\frac{C}{R^2} \int_{B_{3R}} |D{\psi}|^p dx + C \int_{B_{3R}} |D^2 {\psi}|^p dx\nonumber\\
 & \ \ \ \ +C \left(\int_{B_{3R}}k_\e(x)^\frac{2m(p+2)}{2m-p} dx \right)^\frac{2m-p}{m(p+2)}\left(\int_{B_{2R}}\eta^{p+2}(1+|Du_{\e,m}|)^{\frac{m}{m+1}(p+2)}\dx \right)^{\frac{m+1}{m} \frac{p}{p+2}},\nonumber
\end{align}
and so, by using the left inequality in \eqref{IN1}, we obtain
\begin{align}
    &\int_{B_{2R}}\eta^{p+2} (1+|Du_{\e,m}|^2)^\frac{p-2}{2}|D^2u_{\e,m}|^2\dx\nonumber\\
    & \le  \dfrac{C}{R^2} \int_{3R} (1+|Du_{\e,m}|^p)\dx
 +\frac{C}{R^2} \int_{B_{3R}} |D{\psi}|^p dx + C \int_{B_{3R}} |D^2 {\psi}|^p dx\nonumber\\
 & \ \ \ \ +C \left(\int_{B_{3R}}k_\e(x)^\frac{2m(p+2)}{2m-p} dx \right)^\frac{2m-p}{m(p+2)}\left(\int_{B_{2R}}\eta^{p+2}(1+|Du_{\e,m}|)^{\frac{m}{m+1}(p+2)}\dx \right)^{\frac{m+1}{m} \frac{p}{p+2}}.\label{tauem}
\end{align}
Since the function $\phi= \eta^\frac{m+1}{m} \in \mathcal{C}^1_0(B_{2R})$, we can use this choice of $\phi $ together with $v=u_{\e,m}$ in the inequality \eqref{2.1GP} of Lemma \ref{lemma5_GPdN}. Thus, we get 
\begin{align}\label{2.1GP*}
			\int_{B_{2R}}&\eta^{p+2}|Du_{\e,m}|^{\frac{m}{m+1}(p+2)} dx\notag\\
			\le& \ \dfrac{C(p)}{R^\frac{2m}{m+1}}\left(\int_{B_{2R}}|u_{\e,m}|^{2m}dx\right)^\frac{1}{m+1}\left(\int_{B_{2R}}\left|Du_{\e,m}\right|^p dx\right)^\frac{m}{m+1}\notag\\
	& + {C(n,p)}\left(\int_{B_{2R}}|u_{\e,m}|^{2m}dx\right)^\frac{1}{m+1}\left(\int_{B_{2R}}\eta^{p+2}(1+|Du_{\e,m}|^2)^\frac{p-2}{2}\left|D^2u_{\e,m}\right|^2 dx\right)^\frac{m}{m+1},
		\end{align}
  where we also used that $p \ge 2$ and the properties of $\eta$.

Inserting estimate \eqref{tauem} in \eqref{2.1GP*}, we infer
\begin{align*}
			\int_{B_{2R}}&\eta^{p+2}|Du_{\e,m}|^{\frac{m}{m+1}(p+2)} dx\notag\\
			\le& \ \dfrac{C}{R^\frac{2m}{m+1}}\left(\int_{B_{2R}}|u_{\e,m}|^{2m}dx\right)^\frac{1}{m+1}\left(\int_{B_{3R}}(1+\left|Du_{\e,m}\right|^p) dx\right)^\frac{m}{m+1}\notag\\
   &+ \dfrac{C}{R^\frac{2m}{m+1}}\left(\int_{B_{2R}}|u_{\e,m}|^{2m}dx\right)^\frac{1}{m+1}\left(\int_{B_{3R}}\left|D \psi\right|^p dx\right)^\frac{m}{m+1}\notag\\
   &+ {C}\left(\int_{B_{2R}}|u_{\e,m}|^{2m}dx\right)^\frac{1}{m+1}\left(\int_{B_{3R}}\left|D^2 \psi\right|^p dx\right)^\frac{m}{m+1}\notag\\
	& + {C}\left(\int_{B_{2R}}|u_{\e,m}|^{2m}dx\right)^\frac{1}{m+1}\left(\int_{B_{3R}}k_\e(x)^\frac{2m(p+2)}{2m-p} dx \right)^\frac{2m-p}{(m+1)(p+2)} \notag \\
 & \ \ \ \ \cdot \left(\int_{B_{2R}}\eta^{p+2}(1+|Du_{\e,m}|)^{\frac{m}{m+1}(p+2)}\dx \right)^{\frac{p}{p+2}}.
		\end{align*}
Applying Young's inequality in the last term of the preceding estimate, we obtain for some $\theta >0$
\begin{align*}
			\int_{B_{2R}}&\eta^{p+2}|Du_{\e,m}|^{\frac{m}{m+1}(p+2)} dx\notag\\
			\le& \ \dfrac{C}{R^\frac{2m}{m+1}}\left(\int_{B_{2R}}|u_{\e,m}|^{2m}dx\right)^\frac{1}{m+1}\left(\int_{B_{3R}}(1+\left|Du_{\e,m}\right|^p) dx\right)^\frac{m}{m+1}\notag\\
   &+ \dfrac{C}{R^\frac{2m}{m+1}}\left(\int_{B_{2R}}|u_{\e,m}|^{2m}dx\right)^\frac{1}{m+1}\left(\int_{B_{3R}}\left|D \psi\right|^p dx\right)^\frac{m}{m+1}\notag\\
   &+ {C}\left(\int_{B_{2R}}|u_{\e,m}|^{2m}dx\right)^\frac{1}{m+1}\left(\int_{B_{3R}}\left|D^2 \psi\right|^p dx\right)^\frac{m}{m+1}\notag\\
	& + {C_\theta}\left(\int_{B_{2R}}|u_{\e,m}|^{2m}dx\right)^\frac{p+2}{2(m+1)}\left(\int_{B_{3R}}k_\e(x)^\frac{2m(p+2)}{2m-p} dx \right)^\frac{2m-p}{2(m+1)} \notag \\
 & + \theta\int_{B_{2R}}\eta^{p+2}(1+|Du_{\e,m}|)^{\frac{m}{m+1}(p+2)}\dx .
		\end{align*}
Now, using the elementary inequality
$$(1+b)^{\frac{m}{m+1}(p+2)} \le c(p)(1+b^{\frac{m}{m+1}(p+2)}) \quad \forall b \ge 0,$$
we get
\begin{align*}
			\int_{B_{2R}}&\eta^{p+2}(1+|Du_{\e,m}|)^{\frac{m}{m+1}(p+2)} dx\notag\\
			\le& \ \dfrac{C}{R^\frac{2m}{m+1}}\left(\int_{B_{2R}}|u_{\e,m}|^{2m}dx\right)^\frac{1}{m+1}\left(\int_{B_{3R}}(1+\left|Du_{\e,m}\right|^p) dx\right)^\frac{m}{m+1}\notag\\
   &+ \dfrac{C}{R^\frac{2m}{m+1}}\left(\int_{B_{2R}}|u_{\e,m}|^{2m}dx\right)^\frac{1}{m+1}\left(\int_{B_{3R}}\left|D \psi\right|^p dx\right)^\frac{m}{m+1}\notag\\
   &+ {C}\left(\int_{B_{2R}}|u_{\e,m}|^{2m}dx\right)^\frac{1}{m+1}\left(\int_{B_{3R}}\left|D^2 \psi\right|^p dx\right)^\frac{m}{m+1}\notag\\
	& + {C_\theta}\left(\int_{B_{2R}}|u_{\e,m}|^{2m}dx\right)^\frac{p+2}{2(m+1)}\left(\int_{B_{3R}}k_\e(x)^\frac{2m(p+2)}{2m-p} dx \right)^\frac{2m-p}{2(m+1)} \notag \\
 & + \theta c(p)\int_{B_{2R}}\eta^{p+2}(1+|Du_{\e,m}|^{\frac{m}{m+1}(p+2)})\dx +CR^n .
		\end{align*}
Choosing $\theta =\frac{1}{2c(p)}$ and reabsorbing the last integral on the right-hand side of the previous estimate into the
left-hand side, we derive
\begin{align}
			\int_{B_{2R}}&\eta^{p+2}(1+|Du_{\e,m}|)^{\frac{m}{m+1}(p+2)} dx\notag\\
			\le& \ \dfrac{C}{R^\frac{2m}{m+1}}\left(\int_{B_{2R}}|u_{\e,m}|^{2m}dx\right)^\frac{1}{m+1}\left(\int_{B_{3R}}(1+\left|Du_{\e,m}\right|^p) dx\right)^\frac{m}{m+1}\notag\\
   &+ \dfrac{C}{R^\frac{2m}{m+1}}\left(\int_{B_{2R}}|u_{\e,m}|^{2m}dx\right)^\frac{1}{m+1}\left(\int_{B_{3R}}\left|D \psi\right|^p dx\right)^\frac{m}{m+1}\notag\\
   &+ {C}\left(\int_{B_{2R}}|u_{\e,m}|^{2m}dx\right)^\frac{1}{m+1}\left(\int_{B_{3R}}\left|D^2 \psi\right|^p dx\right)^\frac{m}{m+1}\notag\\
	& + {C}\left(\int_{B_{2R}}|u_{\e,m}|^{2m}dx\right)^\frac{p+2}{2(m+1)}\left(\int_{B_{3R}}k_\e(x)^\frac{2m(p+2)}{2m-p} dx \right)^\frac{2m-p}{2(m+1)}  +CR^n . \label{IntEst}
		\end{align}
Inserting estimate \eqref{IntEst} in \eqref{tauem} and using the elementary inequality
$$(b+d)^{\frac{m+1}{m}\frac{p}{p+2}} \le 2^\frac{m+1}{m}c(p) (b^{\frac{m+1}{m}\frac{p}{p+2}}+d^{\frac{m+1}{m}\frac{p}{p+2}}) \le 4 c(p) (b^{\frac{m+1}{m}\frac{p}{p+2}}+d^{\frac{m+1}{m}\frac{p}{p+2}}) \quad \forall b,d \ge 0,$$
we find that
\begin{align}
    &\int_{B_{R}}(1+|Du_{\e,m}|^2)^\frac{p-2}{2}|D^2u_{\e,m}|^2\dx\nonumber\\
    & \le  \dfrac{C}{R^2} \int_{3R} (1+|Du_{\e,m}|^p)\dx
 +\frac{C}{R^2} \int_{B_{3R}} |D{\psi}|^p dx + C \int_{B_{3R}} |D^2 {\psi}|^p dx\nonumber\\
 & \ \ \ \ +\dfrac{C}{R^\frac{2p}{p+2}}
 \left(\int_{B_{3R}}k_\e(x)^\frac{2m(p+2)}{2m-p} dx \right)^\frac{2m-p}{m(p+2)}
\left(\int_{B_{2R}}|u_{\e,m}|^{2m}dx\right)^{\frac{1}{m}\frac{p}{p+2}}\left(\int_{B_{3R}}(1+\left|Du_{\e,m}\right|^p) dx\right)^\frac{p}{p+2}\notag\\
   & \ \ \ \ + \dfrac{C}{R^\frac{2p}{p+2}}
    \left(\int_{B_{3R}}k_\e(x)^\frac{2m(p+2)}{2m-p} dx \right)^\frac{2m-p}{m(p+2)}
\left(\int_{B_{2R}}|u_{\e,m}|^{2m}dx\right)^{\frac{1}{m}\frac{p}{p+2}}\left(\int_{B_{3R}}\left|D \psi\right|^p dx\right)^\frac{p}{p+2}\notag\\
   &  \ \ \ \ + {C}  \left(\int_{B_{3R}}k_\e(x)^\frac{2m(p+2)}{2m-p} dx \right)^\frac{2m-p}{m(p+2)}\left(\int_{B_{2R}}|u_{\e,m}|^{2m}dx\right)^{\frac{1}{m}  \frac{p}{p+2}}\left(\int_{B_{3R}}\left|D^2 \psi\right|^p dx\right)^\frac{p}{p+2}\notag\\
	& \ \ \ \ + {C}  \left(\int_{B_{3R}}k_\e(x)^\frac{2m(p+2)}{2m-p} dx \right)^\frac{2m-p}{2m}\left(\int_{B_{2R}}|u_{\e,m}|^{2m}dx\right)^\frac{p}{2m}\notag\\
 & \ \ \ \ +CR^{n\frac{m+1}{m} \frac{p}{p+2}} \left(\int_{B_{3R}}k_\e(x)^\frac{2m(p+2)}{2m-p} dx \right)^\frac{2m-p}{m(p+2)} , \notag
\end{align}
where we also used that $\eta=1$ on $B_R$.

By virtue of \eqref{disug1} and \eqref{uemS}, we get
\begin{align}
    &\int_{B_{R}}(1+|Du_{\e,m}|^2)^\frac{p-2}{2}|D^2u_{\e,m}|^2\dx\nonumber\\
    & \le  \dfrac{C}{R^2} \int_{3R} (1+|Du|^p)\dx
 +\frac{C}{R^2} \int_{B_{3R}} |D{\psi}|^p dx + C \int_{B_{3R}} |D^2 {\psi}|^p dx\nonumber\\
 & \ \ \ \ +\dfrac{C a^\frac{2p}{p+2}}{R^\frac{p(2m-n)}{m(p+2)}}
 \left(\int_{B_{3R}}k_\e(x)^\frac{2m(p+2)}{2m-p} dx \right)^\frac{2m-p}{m(p+2)}
\left(\int_{B_{3R}}(1+\left|Du\right|^p) dx\right)^\frac{p}{p+2}\notag\\
& \ \ \ \ + \dfrac{C }{R^\frac{2p}{p+2}}
 \left(\int_{B_{3R}}k_\e(x)^\frac{2m(p+2)}{2m-p} dx \right)^\frac{2m-p}{m(p+2)}
\left(\int_{B_{3R}}(1+\left|Du\right|^p) dx\right)^{\frac{m+1}{m}\frac{p}{p+2}}\notag\\
   & \ \ \ \ + \dfrac{Ca^\frac{2p}{p+2}}{R^\frac{p(2m-n)}{m(p+2)}}
    \left(\int_{B_{3R}}k_\e(x)^\frac{2m(p+2)}{2m-p} dx \right)^\frac{2m-p}{m(p+2)}
\left(\int_{B_{3R}}\left|D \psi\right|^p dx\right)^\frac{p}{p+2}\notag\\
 & \ \ \ \ + \dfrac{C}{R^\frac{2p}{p+2}}
    \left(\int_{B_{3R}}k_\e(x)^\frac{2m(p+2)}{2m-p} dx \right)^\frac{2m-p}{m(p+2)}
\left(\int_{B_{2R}}(1+|Du|^p)dx\right)^{\frac{1}{m}\frac{p}{p+2}}\left(\int_{B_{3R}}\left|D \psi\right|^p dx\right)^\frac{p}{p+2}\notag\\
   &  \ \ \ \ + {Ca^\frac{2p}{p+2}R^{n \frac{p}{m(p+2)}}}  \left(\int_{B_{3R}}k_\e(x)^\frac{2m(p+2)}{2m-p} dx \right)^\frac{2m-p}{m(p+2)}\left(\int_{B_{3R}}\left|D^2 \psi\right|^p dx\right)^\frac{p}{m(p+2)}\notag\\
&  \ \ \ \ + {C}  \left(\int_{B_{3R}}k_\e(x)^\frac{2m(p+2)}{2m-p} dx \right)^\frac{2m-p}{m(p+2)}\left(\int_{B_{2R}}(1+|Du|^p)dx\right)^{\frac{1}{m}  \frac{p}{p+2}}\left(\int_{B_{3R}}\left|D^2 \psi\right|^p dx\right)^\frac{p}{p+2}\notag\\
	& \ \ \ \ + {C}a^\frac{2p}{p+2}R^{n \frac{p}{2m}}  \left(\int_{B_{3R}}k_\e(x)^\frac{2m(p+2)}{2m-p} dx \right)^\frac{2m-p}{2m}\notag\\
& \ \ \ \ + {C}  \left(\int_{B_{3R}}k_\e(x)^\frac{2m(p+2)}{2m-p} dx \right)^\frac{2m-p}{2m}\left(\int_{B_{2R}}(1+|Du|^p)dx\right)^\frac{p}{2m}\notag\\
 & \ \ \ \ +CR^{n\frac{m+1}{m} \frac{p}{p+2}} \left(\int_{B_{3R}}k_\e(x)^\frac{2m(p+2)}{2m-p} dx \right)^\frac{2m-p}{m(p+2)} , \label{ST}
\end{align}
where $C:=C(n,p,l,L,\tilde{l},\tilde{\nu})$ is a positive constant independent of $m$ and $\varepsilon$. Using estimate \eqref{SecDer} with $v= u_{\e,m}$ and \eqref{disug1}, we have
\begin{equation*}
 	\int_{B_{R}}|D^2u_{\e,m}|^2\dx\le {c}\int_{B_{3R}}(1+|Du|)^p\dx+c\int_{B_{3R}}|D^2{\psi}|^p\dx+{c}\int_{B_{3R}}|D{\psi}|^p\dx,
\end{equation*}
for a constant $c:=c(n,p,l,L,\tilde{l},\tilde{\nu},M_\varepsilon)$, with $M_\varepsilon:= \sup_{B_r}k_\varepsilon$. This implies that the sequence $(Du_{\e,m})$ is bounded in $W^{1,2}(B_R)$, and so, recalling \eqref{WeakConv}, we have that as $m \to \infty$
$$D u_{\e,m} \rightharpoonup  D u_\varepsilon \quad \text{weakly in } W^{1,2}(B_R)$$
and
$$D u_{\e,m} \to  D u_\varepsilon \quad \text{strongly in } L^{2}(B_R).$$
Hence, by Severini-Egorov theorem, we have that  $(1+|Du_{\e,m}|^2)^\frac{p-2}{4}$ strongly converges to $(1+|Du_{\e}|^2)^\frac{p-2}{4}$ in $L^q(B_R)$ as $m \to \infty$, for every $q < \frac{2p}{p-2}$. In particular, since $2 < \frac{2p}{p-2}$, as $m \to \infty$ 
$$(1+|Du_{\e,m}|^2)^\frac{p-2}{4} \to (1+|Du_{\e}|^2)^\frac{p-2}{4} \quad \text{strongly in } L^2(B_R) 
$$
and then
$$(1+|Du_{\e,m}|^2)^\frac{p-2}{4} D^2u_{\e,m} \rightharpoonup (1+|Du_{\e}|^2)^\frac{p-2}{4}D^2u_{\e} \quad \text{weakly in } L^1(B_R) .
$$
By weak lower semicontinuity, passing to the limit as $m \to \infty$ in \eqref{ST}, we obtain
\begin{align}\label{ST2}
    &\int_{B_{R}}(1+|Du_{\e}|^2)^\frac{p-2}{2}|D^2u_{\e}|^2\dx\nonumber\\
    & \le  \dfrac{C}{R^2} \int_{3R} (1+|Du|^p)\dx
 +\frac{C}{R^2} \int_{B_{3R}} |D{\psi}|^p dx + C \int_{B_{3R}} |D^2 {\psi}|^p dx\nonumber\\
 & \ \ \ \ +\dfrac{C (1+a^\frac{2p}{p+2})}{R^\frac{2p}{p+2}}
 \left(\int_{B_{3R}}k_\e(x)^{p+2} dx \right)^\frac{2}{p+2}
\left(\int_{B_{3R}}(1+\left|Du\right|^p) dx\right)^\frac{p}{p+2}\notag\\
   & \ \ \ \ + \dfrac{C(1+a^\frac{2p}{p+2})}{R^\frac{2p}{p+2}}
    \left(\int_{B_{3R}}k_\e(x)^{p+2} dx \right)^\frac{2}{p+2}
\left(\int_{B_{3R}}\left|D \psi\right|^p dx\right)^\frac{p}{p+2}\notag\\
   &  \ \ \ \ + {C(1+a^\frac{2p}{p+2})}  \left(\int_{B_{3R}}k_\e(x)^{p+2} dx \right)^\frac{2}{p+2}\notag\\
&  \ \ \ \ + {C}  \left(\int_{B_{3R}}k_\e(x)^{p+2} dx \right)^\frac{2}{p+2} \left(\int_{B_{3R}}\left|D^2 \psi\right|^p dx\right)^\frac{p}{p+2}\notag\\
 & \ \ \ \ +CR^{n\frac{p}{p+2}} \left(\int_{B_{3R}}k_\e(x)^{p+2} dx \right)^\frac{2}{p+2} . %\notag
\end{align}
Moreover, using \eqref{disug1} and \eqref{uemS} in \eqref{IntEst} and taking the limit as $m \to \infty$, we infer
\begin{align}\label{ST3}
			\int_{B_{R}}&(1+|Du_{\e}|)^{p+2} dx\notag\\
			\le& \ \dfrac{Ca^2}{R^2}\int_{B_{3R}}(1+\left|Du\right|^p) dx+ \dfrac{Ca^2}{R^2}\int_{B_{3R}}\left|D \psi\right|^p dx+ {Ca^2} \int_{B_{3R}}\left|D^2 \psi\right|^p dx\notag\\
	& + {Ca^{p+2}}\int_{B_{3R}}k_\e(x)^{p+2}dx   +CR^n . 
 \end{align}

\noindent {\bf Step 3.} Let $H_p$ be the function defined in \eqref{Auxiliary_functions}. Since $k_\e \rightarrow k$ strongly in $L^{p+2}(B_r)$ and $p\ge 2$, in view of \eqref{IN2}, from \eqref{ST2} we deduce that the sequence $(H_p(Du_\e))$ is bounded in $W^{1,2}_{loc} (B_R)$ and therefore is bounded in $W^{1,2}_{loc} (B_r)$ for the arbitrariness
of the ball $B_R$ and thanks to a simple covering argument. Thus, as $\e \rightarrow 0 $,
$$H_p(Du_\e)\rightharpoonup w \quad \text{weakly in } W^{1,2}_{loc}(B_r)$$
and so
$$H_p(Du_\e) \to w \quad \text{strongly in } L^{2}_{loc}(B_r) \text{ and a.e.}$$
up to subsequences. We recall that $u_\e\rightharpoonup v$ in $W^{1,p}(B_r)$ and the limit function $v$ still belongs to $ \mathcal{K}_{\tilde{\psi}}$, since this set is convex
and closed. Therefore, by $|H_p(Du_\e)|^2 = |Du_\e|^p$ we get $ w = |Dv|^{\frac{p-2}{2}} Dv$ and
$$ Du_\e \to Dv \quad \text{strongly in } L^{p}_{loc}(B_r).$$
Now we prove that $u\equiv v$ on $B_r$. From weak lower semicontinuity, it holds that for every $\tilde{r}<r$
$$ \int_{B_{\tilde{r}}}F(x,Du) \le \liminf_{\e \to 0}  \int_{B_{\tilde{r}}}F(x,Du_\e) .$$
In order to prove that for every $\tilde{r}<r$
\begin{equation}\label{ST4}
  \lim_{\e \to 0}\biggm|  \int_{B_{\tilde{r}}}F_\e(x,Du_\e) - F(x,Du_\e) \dx \biggm|=0,  
\end{equation}
we observe that

\begin{align}\label{ST5}
&\int_{B_{\tilde{r}}}F_\e (x,Du_\e) - F(x,Du_\e) \dx \nonumber\\
=& \int_{B_{\tilde{r}}}(F_\e(x,Du_\e) - F_\e(x,0)+F_\e(x,0)-F(x,0)+F(x,0)- F(x,Du_\e) )\dx \nonumber\\
=& \int_{B_{\tilde{r}}} \biggm[ \int_0^1 D_\xi F_\e(x, tDu_\e)Du_\e dt - \int_0^1 D_\xi F(x, tDu_\e) Du_\e dt  \biggm] dx + \int_{B_{\tilde{r}}}  [F_\e(x,0)-F(x,0)] dx\nonumber\\
=& \int_{B_{\tilde{r}}} \biggm[ \int_0^1 \biggm\{  D_\xi F_\e(x, tDu_\e) - \biggm( \int_{B_1} \phi (\omega) d\omega\biggm) D_\xi F(x, tDu_\e)\biggm \} Du_\e dt  \biggm] dx + \int_{B_{\tilde{r}}}  [F_\e(x,0)-F(x,0)] dx\nonumber\\
=&\int_{B_{\tilde{r}}} \int_0^1 \int_{B_1} \phi (\omega) (D_\xi F(x+\e \omega, tDu_\e) -  D_\xi F(x, tDu_\e))Du_\e d\omega dt dx + \int_{B_{\tilde{r}}}  [F_\e(x,0)-F(x,0)] dx\nonumber\\
=&:I_{1,\e}+I_{2,\e}
\end{align}
where we used that $\int_{B_1} \phi (\omega) d\omega = 1$ and the properties of the derivatives of the mollifications. By virtue of \eqref{F5*} and  H\"older's inequality we have

\begin{align*}
|I_{1,\e}|&\le \e \int_{B_{\tilde{r}}} \biggm( (k_\e(x)+k(x)) \int_0^1(1+|tDu_\e|)^{p-1} dt\biggm) |Du_\e| dx\\
&\le c \e \int_{B_{\tilde{r}}} ( (k_\e(x)+k(x)) |Du_\e|dx+ c \e \int_{B_{\tilde{r}}} ( (k_\e(x)+k(x)) |Du_\e|^p dx\\
&\le c \e \biggm(\int_{B_{\tilde{r}}} ( (k_\e(x)+k(x))^{p+2} dx\biggm)^{\frac 1{p+2}} \biggm(\int_{B_{\tilde{r}}} (1+|Du_\e|)^{ \frac{p(p+2)}{p+1}} dx \biggm)^{\frac{p+1}{p+2}}.
\end{align*} 
Since $k_\e \rightarrow k$ strongly in $L^{p+2}(B_r)$, in view of \eqref{ST3}, 
 we deduce that the sequence $Du_\e$ is bounded in $B_r$ for the arbitrariness
of the ball $B_R$ and thanks to a simple covering argument. Thus,
\begin{equation}\label{I_1,e}
    |I_{1,\e}|\le c\e
\end{equation}
for a constant $c$ independent of $\e$. On the other hand, {since $F(x,0) \in L^\infty(B_{\tilde{r}})$ by assumption \eqref{eqf2}}, it is immediate to verify that

\begin{equation}\label{I_2,e}
    \lim_{\e \to 0} I_{2,\e}=0.
\end{equation}
Combining \eqref{ST5} with \eqref{I_1,e} and \eqref{I_2,e}, we get \eqref{ST4}. Hence, using the minimality of $u_{\e,m}$ and that $F_\e(x, Du) \to F(x, Du)$ in $L^1(B_r)$ as $\e \to 0$ we obtain

\begin{align*}
\int_{B_{\tilde{r}}} F(x,Dv) & \le \liminf_{\e \to 0} \int_{B_{\tilde{r}}} F(x, Du_\e)= \liminf_{\e \to 0} \int_{B_{\tilde{r}}} F_\e(x, Du_\e)\\
& \le \liminf_{\e \to 0} \liminf_{m \to \infty} \int_{B_{\tilde{r}}} F_\e(x, Du_{\e,m})\\
& \le \liminf_{\e \to 0} \liminf_{m \to \infty} \int_{B_{r}} (F_\e(x, Du_{\e,m})+(u_{\e,m}- \tilde{\psi}-a)_+^{2m})\\
& \le \lim_{\e \to 0} \int_{B_{r}} F_\e(x, Du)= \int_{B_{r}} F(x, Du)
\end{align*}
Letting $ \tilde{r} \to r$ in the previous inequality, we get

$$\int_{B_{r}} F(x,Dv) \le \int_{B_{r}} F(x, Du).$$
The strict convexity of the integrand $F(x, \xi)$ with respect to the $\xi$-variable implies $u \equiv v $ in $B_r$. Therefore, by passing to the limits (first as $ \e \to 0$, and then as $a \to 2 \| u\|_{L^\infty(B_r)}$) in \eqref{ST2}, we infer that
\begin{align*}
    &\int_{B_{R}}|D(V_p(Du))|^2\dx\nonumber\\
    & \le  \dfrac{C}{R^2} \int_{3R} (1+|Du|^p)\dx
 +\frac{C}{R^2} \int_{B_{3R}} |D{\psi}|^p dx + C \int_{B_{3R}} |D^2 {\psi}|^p dx\nonumber\\
 & \ \ \ \ +\dfrac{C (1+\| u\|_{L^\infty}^\frac{2p}{p+2})}{R^\frac{2p}{p+2}}
 \left(\int_{B_{3R}}k(x)^{p+2} dx \right)^\frac{2}{p+2}
\left(\int_{B_{3R}}(1+\left|Du\right|^p) dx\right)^\frac{p}{p+2}\notag\\
   & \ \ \ \ + \dfrac{C(1+\| u\|_{L^\infty}^\frac{2p}{p+2})}{R^\frac{2p}{p+2}}
    \left(\int_{B_{3R}}k(x)^{p+2} dx \right)^\frac{2}{p+2}
\left(\int_{B_{3R}}\left|D \psi\right|^p dx\right)^\frac{p}{p+2}\notag\\
   &  \ \ \ \ + {C(1+\| u\|_{L^\infty}^\frac{2p}{p+2})}  \left(\int_{B_{3R}}k(x)^{p+2} dx \right)^\frac{2}{p+2}\notag\\
&  \ \ \ \ + {C}  \left(\int_{B_{3R}}k(x)^{p+2} dx \right)^\frac{2}{p+2} \left(\int_{B_{3R}}\left|D^2 \psi\right|^p dx\right)^\frac{p}{p+2}\notag\\
 & \ \ \ \ +CR^{n\frac{p}{p+2}} \left(\int_{B_{3R}}k(x)^{p+2} dx \right)^\frac{2}{p+2}  %\notag
\end{align*}
i.e. the conclusion.
\section{The Lipschitz continuity}\label{LIP}

\subsection{An Approximation Result}\label{Apres}
The main tool to prove Theorem \ref{Mthm} is Lemma \ref{apprlem1}, which allows us to approximate from below the function $F$ with a sequence of functions $(F_j)$ satisfying standard $p$-growth conditions (see \cite[Proposition 4.1]{cupini.guidorzi.mascolo2003}).

\begin{lem}\label{apprlem1}
Let $F : \Omega \times \mathbb{R}^n \rightarrow [0,+\infty)$, $F=F(x,\xi)$, be a Carath\'{e}odory function,
convex with respect to $\xi$, satisfying assumptions \eqref{eqf1}--\eqref{eqf5}. Then there exists
a sequence $(F_j)$ of Carath\'{e}odory functions $F_j: \Omega \times \mathbb{R}^n \rightarrow [0,+\infty)$, convex with respect to the last variable, monotonically convergent to $F$, such that for a.e.\ $x \in \Omega$ and every $\xi \in \mathbb{R}^n$,
\begin{itemize}
\item[(i)] 
$F_j(x,\xi) = \tilde{F}_j (x,|\xi|)$,
\item[(ii)] 
$F_j(x,\xi)\leq F_{j+1}(x, \xi) \leq F(x,\xi)$,  $ \forall j \in\mathbb{N}$,
\item[(iii)]
$\langle D_{\xi\xi} F_j(x,\xi)\lambda, \lambda\rangle \geq \nu_1(1+|\xi|^2)^\frac{p-2}{2}|\lambda|^2$, $\forall \lambda \in \mathbb{R}^n$,
with $\nu_1=\nu_1(\nu,p)$,
\item[(iv)] there exist $K_0$, independent of $j$, and $K_1$, depending on $j$, such that
\begin{align*}
& K_0(|\xi|^p-1) \leq F_j(x,\xi)\leq L(1+ |\xi|^q),\\
&F_j (x,\xi)\leq K_1(1 + |\xi|^p),
\end{align*} 
\item[(v)]there exists a constant $C=C(p,q,j) > 0$ such that
\begin{align*}
&|D_{x \xi} F_j(x,\xi)| \leq Lg(x)(1 +|\xi|^2)^\frac{q-1}{2},\\
&|D_{x \xi} F_j(x,\xi)| \leq  C g(x)(1 +|\xi|^2)^\frac{p-1}{2}.
\end{align*}
\end{itemize}
\end{lem}

\subsection{The Approximating Problems}\label{Approb}
Fix a compact set $\Omega' \Subset \Omega$ and let $a \ge  \Vert u- \psi \Vert_{L^\infty(\Omega')} $ and $m > p/2$. We consider $u_j \in W^{1,p}(\Omega') \cap L^{2m}(\Omega')$ solution to the problem
\begin{equation}\label{OP}
    \min \left\{ \int_{\Omega'} \left( F_j(x,Dv)+(v-\psi-a)_+^{2m} \right) dx \ : \ v \ge \psi, \ v=u \ \text{on} \ \partial \Omega'\right\},
\end{equation}
where $(F_j)$ is the sequence of functions defined in Lemma \ref{apprlem1}. We observe that the solution $u_j$ to \eqref{OP} solves the variational inequality
\begin{equation}\label{VI}
    \int_{\Omega'} \langle D_\xi F_j(x,Du_j), D(\varphi - u_j) \rangle dx + 2m \int_{\Omega'} (u_j-\psi-a)_+^{2m-1}(\varphi-u_j)dx \ge 0,  
\end{equation}
for all $ \varphi \in u+W^{1,p}_0(\Omega')$, with $\varphi \ge \psi$.

 We have the following
\begin{lem}\label{lemmaL}
As $j \to + \infty$, we have that
$$ \int_{\Omega '} ( u_j - \psi-a )^{2m}_+ dx \to 0, \quad \int_{\Omega '} F_j(x,Du_j) dx \to \int_{\Omega '} F(x,Du)dx $$
and
$$u_j \to u \text{  \  strongly in \ } W^{1,p}(\Omega').$$
    
\end{lem}
\begin{proof}
By the minimality of $u_j$ and using $u$ as an admissible test function, we get
\begin{equation}\label{en1}
    \int_{\Omega '} \left( F_j(x,Du_j) +(u_j-\psi-a)^{2m}_+\right) dx \le \int_{\Omega '} F_j(x,Du)dx,
\end{equation}
where we also used that $u-\psi \le a$ a.e.\ in $\Omega'$. Moreover, the growth assumption at (iv) in Lemma \ref{apprlem1} and the monotonicity of $(F_j)$ imply
\begin{align}\label{en2}
    K_0 \int_{\Omega '} (|Du_j|^p-1) dx \le & \int_{\Omega '} \left( F_j(x,Du_j)+(u_j-\psi-a)^{2m}_+ \right)dx \notag\\
    \le & \int_{\Omega '} F_j(x,Du)dx \notag\\
    \le & \int_{\Omega '} F(x,Du)dx.
\end{align}
Hence, the sequence $(Du_j)$ is bounded in $L^p(\Omega')$ and there exists $w \in W^{1,p}(\Omega')$ such that
$$u_j \rightharpoonup w \text{ \ weakly in \ } u+W^{1,p}_0(\Omega') \text{ \ as \ } j \to + \infty,$$
and $w \ge \psi$, since $\{ v \ge \psi \}$ is convex and closed. 
Passing to the limit as $j \to + \infty$ in \eqref{en1} and taking into account \eqref{en2}, we have
\begin{equation}\label{eqn}
    \limsup_{j \to + \infty} \int_{\Omega'} \left(   F_j(x,Du_j)+ (u_j-\psi-a)^{2m}_+\right)dx \le \int_{\Omega'} F(x,Du)dx.
\end{equation}
On the other hand, by the weak lower semicontinuity, for every $j_0 \in \mathbb{N}$, we have
\begin{align}
     \int_{\Omega'} F_{j_0} (x,Dw)dx \le & \liminf_j  \int_{\Omega'} F_{j_0}(x,Du_j)dx \notag\\
     \le & \liminf_j  \int_{\Omega'} F_j(x,Du_j)dx \notag\\
     \le & \liminf_j \int_{\Omega'}\left( F_j(x,Du_j)+ (u_j- \psi-a)_+^{2m} \right)dx \notag\\ 
     \le &  \int_{\Omega'} F(x,Du)dx, \label{en3}
     \end{align}
where we also used the monotonicity of $(F_j)$ and \eqref{eqn}. Letting $j_0 \to + \infty$ in \eqref{en3}, by the monotone convergence theorem and the minimality of $u$, we obtain  
\begin{equation}\label{en4}
     \int_{\Omega'} F(x,Dw)dx = \liminf_{j_0}  \int_{\Omega'} F_{j_0}(x,Dw) dx \le  \int_{\Omega'} F(x,Du)dx \le  \int_{\Omega'} F(x,Dw)dx.
     \end{equation}
     Therefore, by the minimality of $u$ and the strict convexity of $F$, we get $u=w$ a.e.\ on $\Omega'$. 

\noindent Combining inequalities \eqref{eqn}, \eqref{en3} and \eqref{en4}, we get
\begin{equation}\label{en5}
    \lim_j  \int_{\Omega'} (u_j-\psi-a)^{2m}_+ dx=0 \text{ \ and \ } \lim_j  \int_{\Omega'} F_j(x,Du_j)dx =  \int_{\Omega'}F(x,Du)dx.
\end{equation}
Moreover,
\begin{equation}\label{en6}
    \sup_j  \int_{\Omega'} |u_j|^{2m}dx \le 2^m \left(1+\int_{\Omega'}|\psi+a|^{2m}dx \right).
\end{equation}
Now, Taylor's formula and strict convexity of $F_j$ yield
\begin{align*}
    \nu & \int_{\Omega'} (1+|Du|+|Du_j|^2)^\frac{p-2}{2}|Du-Du_j|^2 dx \notag\\
    & \le \int_{\Omega'} \left( F_j(x,Du)-F_j(x,Du_j)+ \langle D_\xi F_j(x,Du_j),Du_j-Du\rangle \right)dx.
\end{align*}
Since $u_j$ satisfies \eqref{VI} with $\varphi =u$, from the previous inequality we deduce that
\begin{align*}
   \nu & \int_{\Omega'} (1+|Du|+|Du_j|^2)^\frac{p-2}{2}|Du-Du_j|^2 dx \notag\\
    & \le \int_{\Omega'} \left( F_j(x,Du)-F_j(x,Du_j) \right)dx+ 2m \int_{\Omega'} (u_j-\psi-a)_+^{2m-1}(u-u_j)dx \notag\\
    & \le \int_{\Omega'} \left( F(x,Du)-F_j(x,Du_j) \right)dx+ 2m \int_{\Omega'} (u_j-\psi-a)_+^{2m-1}(u-u_j)dx ,
\end{align*}
where in the last line we used again that $F_j(x,\xi) \le F(x,\xi)$, for a.e.\ $x \in \Omega $ and every $\xi \in \mathbb{R}^n$.
Therefore, by \eqref{en5} and \eqref{en6}, we get
$$\limsup_j \int_{\Omega'} (1+|Du|+|Du_j|^2)^\frac{p-2}{2}|Du-Du_j|^2 dx =0, $$
and so
$$Du_j \to Du \text{ \ strongly in \ } L^{p}(\Omega'),$$
which concludes the proof.
\end{proof}

\subsection{Proof of Theorem \ref{Mthm}}
This section is devoted to the proof of Theorem \ref{Mthm}, which is divided in several steps.

\vskip0.5cm

\noindent \textbf{Step 1.} The first step consists in the linearization argument, i.e.\
we shall show that the solution $u_j$ to \eqref{OP} solves a suitable elliptic equation, as proved in the following

\begin{thm}\label{LT}
   Let $u_j \in W^{1,p}(\Omega') \cap L^{2m}(\Omega')$ be the solution to \eqref{OP}, under assumptions \eqref{eqf1}--\eqref{eqf5}. Assume that $2 \le p < q$ satisfy \eqref{gap}.
   If $\psi \in W^{2, r}_{loc}(\Omega)$, then there exists $h_j \in L^r_{loc}(\Omega') $ such that
    \begin{equation}\label{LEq}
        - \ \mathrm{ div } D_\xi F_j(x,Du_j)=h_j \text{ \ a.e.\ in \ } \Omega,
    \end{equation}
   where we denote
   $$h_j = \mathrm{div } (D_\xi F_j(x,Du_j)) \chi_{\{  u_j = \psi \}} .$$
   \end{thm}
\begin{proof}
    Given $\varepsilon >0$, let $k_\varepsilon : (0,+\infty) \to [0,1]$ be a smooth function such that $k'_\varepsilon(s) \le 0$ and
    $$k_\varepsilon(s)=
    \begin{cases}
         1 & \quad s \le \varepsilon, \\
        0 & \quad s \ge 2\varepsilon.
    \end{cases}$$
    Testing the variational inequality \eqref{VI} with
    $$\varphi=u_j+t \ \eta \ k_\varepsilon(u_j-\psi),$$
    where $\eta \in \mathcal{C}^1_0(\Omega')$, $\eta \ge 0$ and $0 < t < 1$, we get
    \begin{equation*}
    \int_{\Omega'} \langle D_\xi F_j(x,Du_j), D(\eta  k_\varepsilon( u_j-\psi)) \rangle dx + 2m \int_{\Omega'} (u_j-\psi-a)_+^{2m-1}\eta k_\varepsilon(u_j-\psi)dx \ge 0.  
\end{equation*}
Since $s \mapsto k_\varepsilon(s)$ has support in $[0,2 \varepsilon]$, choosing $0 < \varepsilon < \frac{1}{2}$ and $a >>1$, we have
$$\int_{\Omega'} (u_j-\psi-a)^{2m-1}_+ \eta k_\varepsilon(u_j-\psi)dx=0.$$
Therefore, for every $\eta \in \mathcal{C}^1_0(\Omega')$, with $\eta \ge 0$,
\begin{equation*}
    \int_{\Omega'} \langle D_\xi F_j(x,Du_j), D(\eta  k_\varepsilon( u_j-\psi)) \rangle dx  \ge 0.  
\end{equation*}
At this point, {owing to the higher differentiability property of $u_j$ obtained in Theorem \ref{highdiff}, we can argue as in \cite[Theorem 4.1]{cepdn} and we get the desired conclusion.}
\end{proof}

\begin{rmk}\label{remark}
    The sequence $(h_j)$  defined in Theorem \ref{LT} is uniformly bounded in $L^r_{loc}(\Omega')$. Indeed, by \eqref{eqf4} and \eqref{eqf5},
    \begin{align}
        |h_j(x)| = & \ |\mathrm{div } (D_\xi F_j(x,Du_j)) \chi_{\{  u_j = \psi \}}| \notag \\
        = & \ |\mathrm{div } (D_\xi F_j(x,D \psi)) \chi_{\{  u_j = \psi \}}| \notag\\
        \le & \ |D_{x \xi} F_j(x,D \psi)|+ |D_{\xi \xi} F_j(x,D \psi)|  |D^2 \psi| \notag \\
        \le & \  g(x)(1+|D \psi|^2)^\frac{q-1}{2} + \tilde{L} (1+|D \psi|^2)^\frac{q-2}{2} |D^2 \psi|. \label{remark1}
    \end{align}
    Since the assumption $\psi \in W^{2, r}_{loc}(\Omega)$, with $r > \max \{ n , p+2 \}$, implies $D \psi \in L^\infty_{loc}(\Omega)$, we get
    \begin{equation*}
        \Vert h_j \Vert_{L^r(\Omega')} \le C(1+\Vert D\psi \Vert_{L^\infty(\Omega')}^{q-1}) \left(\Vert g \Vert_{L^r(\Omega')}+ \Vert D^2 \psi \Vert_{L^r(\Omega')} \right),
    \end{equation*}
    where $C$ is a positive constant independent of $j$.
\end{rmk}

\noindent \textbf{Step 2.} Here, we first establish some higher differentiability estimates for the approximating minimizers with constants independent of the approximating parameters. Next, this allows us to prove a uniform higher integrability result for the gradient of the approximating minimizers.

We start by proving the following Cacciopoli type inequality.
\begin{thm}\label{Caccioppoli}
   Let $u_j \in W^{1,p}(\Omega') \cap L^{2m}(\Omega')$ be the solution to \eqref{OP}, under assumptions \eqref{eqf1}--\eqref{eqf5}. Assume that $2 \le p < q$ satisfy \eqref{gap} and $\psi \in W^{2, r}_{loc}(\Omega)$. Then, the following inequality
   \begin{align}
   &   \int_{\Omega'} \eta^2  (1+|Du_j|^2)^\frac{p-2+\gamma}{2} |D^2u_j|^2 dx \notag \\
   & \le \    C(1+\gamma^2)\int_{\Omega'} \eta^2 (g^2(x)+|D^2 \psi|^2) (1+|Du_j|^2)^\frac{2q-p+\gamma}{2} dx \notag \\
      & \ \ \ \  +  C \int_{\Omega'} |D \eta|^2(1+|Du_j|^2)^\frac{q+\gamma}{2} dx \notag\\
       & \ \ \ \  + C \int_{\Omega'} |D\eta|^2  ( g^2(x)+|D^2 \psi|^2)(1+|Du_j|^2)^\frac{p+\gamma}{2}  dx. \notag
\end{align}
holds true for every $\gamma \ge 0$ and for every $\eta \in \mathcal{C}^{1}_0(\Omega')$, where $C$ is a positive constant independent of $j$.
\end{thm}
\begin{proof}
 We test the equation \eqref{LEq} with the function $D_{x_s} \varphi$, with $s=1,...,n$, thus getting
    \begin{equation*}
        \int_{\Omega'} \langle  D_\xi F_j(x,Du_j), D(D_{x_s} \varphi)\rangle dx = \int_{\Omega'} h_j D_{x_s}\varphi dx,
    \end{equation*}
    for every $\varphi \in \mathcal{C}^{1}_0(\Omega')$. By Theorem \ref{highdiff}, we have that $u_j \in W^{2,2}_{loc}(\Omega')$, therefore integrating by parts the integral in the left hand side in the previous identity, we obtain
    \begin{equation}\label{SV}
        \int_{\Omega'}\left( \sum_{i,l} D_{\xi_i\xi_l}F_j(x,Du_j) D_{x_l x_s}u_j D_{x_i}\varphi  + \sum_{i}D_{\xi_ix_s}F_j(x,Du_j)D_{x_i}\varphi\right)dx= -\int_{\Omega'} h_j D_{x_s} \varphi dx
        \end{equation}
    for every $s \in \{1,...,n\}$ and all $\varphi \in \mathcal{C}^{1}_0(\Omega')$.
 Thanks to the Lipschitz continuity of $u_j$ proved in \cite{DeG} and {the higher differentiability result} given by Theorem \ref{highdiff}, we can test the equation \eqref{SV} with the function
   $$\varphi= \eta^2 (1+|Du_j|^2)^\frac{\gamma}{2}D_{x_s}u_j,$$
   where $\eta \in \mathcal{C}^1_0(\Omega')$ and $\gamma \ge 0$.
    An easy calculation gives that
    \begin{align}
        D_{x_i} \varphi=\ & 2 \eta D_{x_i} \eta (1+|Du_j|^2)^\frac{\gamma}{2} D_{x_s}u_j + \eta^2 (1+|Du_j|^2)^\frac{\gamma}{2} D_{x_s x_i} u_j   \notag\\
         & + {\gamma} \eta^2 (1+|Du_j|^2)^\frac{\gamma-2}{2} |Du_j| D_{x_i}(|Du_j|) D_{x_s}u_j  .
        \label{Der}
    \end{align}
    Inserting \eqref{Der} in \eqref{SV}, we get
    \begin{align*}
        0= \ & 2 \int_{\Omega'} \eta (1+|Du_j|^2)^\frac{\gamma}{2} \sum_{i,l} D_{\xi_i \xi_l}F_j(x,Du_j) D_{x_l x_s}u_j D_{x_i} \eta D_{x_s}u_j dx \\
        & + \int_{\Omega'} \eta^2 (1+|Du_j|^2)^\frac{\gamma}{2} \sum_{i,l} D_{\xi_i \xi_l}F_j(x,Du_j) D_{x_l x_s}u_j  D_{x_s x_i}u_j dx \\
        & + \gamma \int_{\Omega'} \eta^2 (1+|Du_j|^2)^\frac{\gamma-2}{2}   \sum_{i,l} D_{\xi_i \xi_l}F_j(x,Du_j) D_{x_l x_s}u_j |Du_j| D_{x_i}(|Du_j|) D_{x_s }u_j dx \\
 & + 2 \int_{\Omega'} \eta (1+|Du_j|^2)^\frac{\gamma}{2}  \sum_{i} D_{\xi_i x_s}F_j(x,Du_j) D_{x_i}\eta D_{x_s}u_j dx\\
 &+ \int_{\Omega' } \eta^2 (1+|Du_j|^2)^\frac{\gamma}{2}  \sum_{i} D_{\xi_i x_s}F_j(x,Du_j) D_{x_s x_i} u_j dx\\
 & + \gamma \int_{\Omega'} \eta^2 (1+|Du_j|^2)^\frac{\gamma-2}{2}   \sum_{i} D_{\xi_i x_s}F_j(x,Du_j) |Du_j| D_{x_i}(|Du_j|) D_{x_s }u_j dx \\ 
 &  + 2 \int_{\Omega'} \eta (1+|Du_j|^2)^\frac{\gamma}{2}  h_j D_{x_s} \eta D_{x_s} u_j dx \notag \\
 & + \int_{\Omega'} \eta^2 (1+|Du_j|^2)^\frac{\gamma}{2}   h_j D_{x_s x_s} u_j dx \notag \\
 &  + \gamma \int_{\Omega'} \eta^2 (1+|Du_j|^2)^\frac{\gamma-2}{2}  h_j |Du_j| D_{x_s}(|Du_j|) D_{x_s} u_j dx  \\
 =: \ & I_{1,s} +  I_{2,s}+ I_{3,s}+ I_{4,s}+ I_{5,s}+ I_{6,s}+ I_{7,s}+I_{8,s}+I_{9,s}.
    \end{align*}
    We now sum the previous equation with respect to $s$ from $1$ to $n$ and we denote by $I_1$--$I_{9}$
the corresponding integrals. Previous identity yields
\begin{equation}
    I_2+I_3 \le |I_1|+|I_4|+|I_5|+|I_6|+|I_7|+|I_8|+|I_9|.
\end{equation}
The integrals $I_2, I_3, |I_1|, |I_4|, |I_5|$ and $|I_6|$ can be estimated similarly as in the proof of \cite[Lemma 3.6]{EPdN}, thus getting
  \begin{align}
   & \dfrac{\nu_1}{2} \int_{\Omega'} \eta^2 (1+|Du_j|^2)^\frac{p-2+\gamma}{2} |D^2u_j|^2 dx \notag\\
   & \le \    2 \varepsilon \int_{\Omega'} \eta^2 (1+|Du_j|^2)^\frac{p-2+\gamma}{2} |D^2u_j|^2 dx \notag \\
   & \ \ \ \  +  C_\varepsilon(1+\gamma^2)\int_{\Omega'} \eta^2 g^2(x)(1+|Du_j|^2)^\frac{2q-p+\gamma}{2} dx \notag \\
      & \ \ \ \  +  C_\varepsilon\int_{\Omega'} |D \eta|^2  (1+|Du_j|^2)^\frac{q+\gamma}{2} dx \notag \\
   & \ \ \ \ + 2\underbrace{ \int_{\Omega'} \eta  |h_j| |D \eta| |D u_j| (1+|Du_j|^2)^\frac{\gamma}{2} dx}_{=: \ A_1} \notag \\
 & \ \ \ \ + \underbrace{\int_{\Omega'} \eta^2  (1+|Du_j|^2)^\frac{\gamma}{2} | h_j| |D^2 u_j| dx}_{=: \ A_2} \notag \\
 & \ \ \ \ + \gamma\underbrace{ \int_{\Omega'} \eta^2   |h_j| (1+|Du_j|^2)^\frac{\gamma-2}{2} |D(|Du_j|)| |D u_j|^2 dx}_{=: \ A_3} .\label{Es3}
\end{align} 
By using Kato's inequality 
\begin{equation}
    |D(|Du_j|)| \le |D^2u_j|, \label{Kato}
\end{equation}
we find that
\begin{align}
   2 A_1+A_2+ \gamma A_3 \le & \ 2 \int_{\Omega'} \eta  |h_j| |D \eta|  (1+|Du_j|^2)^\frac{q+\gamma}{2} dx \notag\\
   & \   +(1+\gamma) \int_{\Omega'} \eta^2  |h_j|(1+|Du_j|^2)^\frac{q+\gamma-1}{2} |D^2 u_j| dx. \notag
\end{align}
Using Young's inequality in the right hand side of the last estimate, we get
\begin{align}
  2 A_1+A_2+ \gamma A_3 \le & \ \varepsilon   \int_{\Omega'} \eta^2 (1+|Du_j|^2)^\frac{p-2+\gamma}{2} |D^2u_j|^2 dx \notag \\
  & \ + C_\varepsilon (1+\gamma^2) \int_{\Omega'} \eta^2  |h_j|^2(1+|Du_j|^2)^\frac{2q-p+\gamma}{2}  dx \notag\\
  & \ + C_\varepsilon \int_{\Omega'} |D\eta|^2  |h_j|^2(1+|Du_j|^2)^\frac{p+\gamma}{2}  dx. \label{Young}
\end{align}
Inserting \eqref{Young} in \eqref{Es3}, we find that
\begin{align}
   & \dfrac{\nu_1}{2} \int_{\Omega'} \eta^2 (1+|Du_j|^2)^\frac{p-2+\gamma}{2} |D^2u_j|^2 dx \notag\\
   & \le \    3 \varepsilon \int_{\Omega'} \eta^2 (1+|Du_j|^2)^\frac{p-2+\gamma}{2} |D^2u_j|^2 dx \notag \\
   & \ \ \ \  +  C_\varepsilon(1+\gamma^2)\int_{\Omega'} \eta^2 ( g^2(x)+|h_j(x)|^2)(1+|Du_j|^2)^\frac{2q-p+\gamma}{2} dx. \notag \\
   & \ \ \ \  +  C_\varepsilon\int_{\Omega'} |D \eta|^2  (1+|Du_j|^2)^\frac{q+\gamma}{2} dx \notag \\
   & \ \ \ \  + C_\varepsilon \int_{\Omega'} |D\eta|^2  |h_j|^2(1+|Du_j|^2)^\frac{p+\gamma}{2}  dx. \notag
\end{align} 
Choosing $\varepsilon= \frac{\nu_1}{12}$ and reabsorbing the first integral in the right hand side by the left hand side, we infer
\begin{align}
   &  \int_{\Omega'} \eta^2 (1+|Du_j|^2)^\frac{p-2+\gamma}{2} |D^2u_j|^2 dx \notag\\
   & \le \      C(1+\gamma^2)\int_{\Omega'} \eta^2 ( g^2(x)+|h_j(x)|^2)(1+|Du_j|^2)^\frac{2q-p+\gamma}{2} dx. \notag \\
   & \ \ \ \  +  C\int_{\Omega'} |D \eta|^2  (1+|Du_j|^2)^\frac{q+\gamma}{2} dx \notag \\
   & \ \ \ \  + C \int_{\Omega'} |D\eta|^2  |h_j|^2(1+|Du_j|^2)^\frac{p+\gamma}{2}  dx, \notag
\end{align} 
with a constant $C$ independent of $j$. 

By \eqref{remark1} and using that $D \psi \in L^\infty_{loc}(\Omega)$, we obtain 
\begin{equation*}
        | h_j(x) | \le C(\tilde{L})(1+\Vert D\psi \Vert_{L^\infty(\Omega')}^{q-1}) \left( g (x)+ |D^2 \psi(x)| \right), \quad \text{ for a.e.\ } x \in \Omega',
    \end{equation*}
and so
\begin{align}
   & \int_{\Omega'} \eta^2 (1+|Du_j|^2)^\frac{p-2+\gamma}{2} |D^2u_j|^2 dx \notag\\
   & \le C(1+\gamma^2)\int_{\Omega'} \eta^2 ( g^2(x)+|D^2 \psi(x)|^2)(1+|Du_j|^2)^\frac{2q-p+\gamma}{2} dx. \notag \\
   & \ \ \ \  +  C\int_{\Omega'} |D \eta|^2  (1+|Du_j|^2)^\frac{q+\gamma}{2} dx \notag \\
   & \ \ \ \  + C \int_{\Omega'} |D\eta|^2  ( g^2(x)+|D^2 \psi(x)|^2)(1+|Du_j|^2)^\frac{p+\gamma}{2}  dx, \notag
\end{align} 
where $C=C(n,p,q,\nu,l,L,\tilde{L}, \Vert D \psi \Vert_{L^\infty})$,
i.e.\ the conclusion.
\end{proof}

Now, we are in position to prove the higher integrability result.

\begin{thm}
  Let $u_j \in W^{1,p}(\Omega') \cap L^{2m}(\Omega')$ be the solution to \eqref{OP}, under assumptions \eqref{eqf1}--\eqref{eqf5}. Assume that $2 \le p < q$ satisfy \eqref{gap} and $\psi \in W^{2, r}_{loc}(\Omega)$. Then
  $$Du_j \in L_{loc}^{\frac{2rm}{2m+r}(p-q+1)}(\Omega')$$
  and the following estimate
  \begin{equation}\label{HighInt}
      \int_{B_\rho} |Du_j|^{\frac{2rm}{2m+r}(p-q+1)}dx \le \dfrac{C \Lambda^\frac{2rm}{2m+r}}{(R-\rho)^r} \left( 1+\Vert \psi \Vert_{L^\infty(\Omega')}^{2m}+a^{2m} \right)^\frac{r}{2m+r}
  \end{equation}
  holds true for every balls $B_\rho \subset B_R \Subset \Omega'$, where $\Lambda=\Vert 1+g \Vert_{L^r(B_R)}+\Vert D^2 \psi \Vert_{L^r(B_R)}$ and $C$ is a positive constant independent of $j$ and $m$.
\end{thm}
\begin{proof}
By the Caccioppoli inequality at Theorem \ref{Caccioppoli}, together with the Lipschitz regularity of $u_j$ (see \cite{DeG}), we have that
$$(1+|Du_j|^2)^{\frac{p-2}{2}+\gamma} |D^2u_j|^2  \in L^1_{loc}(\Omega'),$$
for every $\gamma >0$. Therefore, applying the inequality \eqref{2.1GP} in Lemma \ref{lemma5_GPdN} with $p$ replaced by $p+2 \gamma$, we get
\begin{align*}
    &\int_{\Omega'} \eta^2 |Du_j|^{\frac{m}{m+1}(p+2+2\gamma)}dx \\
    & \le (p+2+2\gamma)^2 \left(  \int_{\Omega'} \eta^2 |u_j|^{2m} dx\right)^\frac{1}{m+1} \left(  \int_{\Omega'} \eta^2 |D \eta|^2 |Du_j|^{p+2 \gamma} dx\right)^\frac{m}{m+1} \\
    & \ \ \ \ + nN (p+2+2\gamma)^2 \left(  \int_{\Omega'} \eta^2 |u_j|^{2m} dx\right)^\frac{1}{m+1} \left(  \int_{\Omega'} \eta^2 |D u_j|^{p-2+2\gamma} |D^2u_j|^{2 } dx\right)^\frac{m}{m+1}, 
\end{align*}
for every $\eta \in \mathcal{C}^1_0(\Omega)$ such that $0 \le \eta \le 1$. Using Theorem \ref{Caccioppoli} to estimate the last integral in the right hand side of the previous inequality, we obtain
\begin{align*}
    &\int_{\Omega'} \eta^2 |Du_j|^{\frac{m}{m+1}(p+2+2\gamma)}dx \\
    & \le (p+2+2\gamma)^2 \left(  \int_{\Omega'} \eta^2 |u_j|^{2m} dx\right)^\frac{1}{m+1} \left(  \int_{\Omega'} \eta^2 |D \eta|^2 |Du_j|^{p+2 \gamma} dx\right)^\frac{m}{m+1} \\
    & \ \ \ \ + c (p+2+2 \gamma)^4 \left(  \int_{\Omega'} \eta^2 |u_j|^{2m} dx\right)^\frac{1}{m+1} \left(  \int_{\Omega'} \eta^2 ( g^2+|D^2 \psi|^2)(1+|Du_j|^2)^\frac{2q-p+2\gamma}{2} dx\right)^\frac{m}{m+1}\\
    & \ \ \ \  +c (p+2+2 \gamma)^2 \left(  \int_{\Omega'} \eta^2 |u_j|^{2m} dx\right)^\frac{1}{m+1} \left(  \int_{\Omega'} |D\eta|^2 (1+|Du_j|^2)^\frac{q+2\gamma}{2} dx\right)^\frac{m}{m+1} \\
    & \ \ \ \ + c (p+2+2 \gamma)^2 \left(  \int_{\Omega'} \eta^2 |u_j|^{2m} dx\right)^\frac{1}{m+1} \left(  \int_{\Omega'} |D\eta|^2 ( g^2+|D^2 \psi|^2)(1+|Du_j|^2)^\frac{p+2\gamma}{2} dx\right)^\frac{m}{m+1}\\
    & \le c(p+2+2\gamma)^4 \left(  \int_{\Omega'} \eta^2 |u_j|^{2m} dx\right)^\frac{1}{m+1} \left(  \int_{\Omega'} (\eta^2+|D\eta|^2) (1+ g^2+|D^2 \psi|^2)(1+|Du_j|^2)^\frac{2q-p+2\gamma}{2} dx\right)^\frac{m}{m+1},
\end{align*}
where we used that $1+\gamma \le p+2+2\gamma$ and $p+2\gamma \le q + 2 \gamma \le 2q-p+2\gamma$. Applying H\"older's inequality in the right hand side of the previous estimate implies
\begin{align}
     &\int_{\Omega'} \eta^2 |Du_j|^{\frac{m}{m+1}(p+2+2\gamma)}dx  \notag\\
     & \le c(p+2+2\gamma)^4 \left(  \int_{\Omega'} \eta^2 |u_j|^{2m} dx\right)^\frac{1}{m+1} \notag\\
     & \ \ \ \ \cdot \left(  \int_{\Omega'} (\eta^2+|D\eta|^2) (1+ g+|D^2 \psi|)^r dx\right)^\frac{2m}{r(m+1)} \notag\\
     & \ \ \ \ \cdot \left(  \int_{\Omega'} (\eta^2+|D\eta|^2) (1+|Du_j|)^\frac{r(2q-p+2\gamma)}{r-2} dx\right)^\frac{m(r-2)}{r(m+1)}. \label{stima1}
\end{align}
Now, fix concentric balls $B_\rho \subset B_s \subset B_t \subset B_R \Subset \Omega'$ and let $\eta \in \mathcal{C}^1_0(B_t)$ such that $\eta =1$ on $B_s$, $0 \le \eta \le 1$ and $|D \eta | \le \frac{c}{t-s}$. Without loss of generality we may assume that $|B_R | \le 1$. Noting that
$$\eta^2 + |D \eta |^2 \le 1+ \dfrac{c}{(t-s)^2} \le \dfrac{\tilde{c}}{(t-s)^2},$$
since $t-s \le 1$,
inequality \eqref{stima1} yields
\begin{align}
    &\int_{B_s} |Du_j|^{\frac{m}{m+1}(p+2+2\gamma)}dx  \notag\\
     & \le C \dfrac{\left( \Vert 1+g \Vert_{L^r(B_R)}+\Vert D^2 \psi \Vert_{L^r(B_R)} \right)^\frac{2m}{m+1}}{(t-s)^\frac{2m}{m+1}}  \left(  \int_{B_t}  |u_j|^{2m} dx\right)^\frac{1}{m+1} \notag\\
     & \ \ \ \ \cdot \left(  \int_{B_t}  (1+|Du_j|)^\frac{r(2q-p+2\gamma)}{r-2} dx\right)^\frac{m(r-2)}{r(m+1)}. \label{stima2}
\end{align}
We choose $\gamma \ge 0$ such that
$$\dfrac{r(2q-p+2\gamma)}{r-2}=\dfrac{m}{m+1}(p+2+2\gamma) \Longleftrightarrow 2 \gamma = \dfrac{2mr(p-q+1)-2m(p+2)-r(2q-p)}{2m-r} $$
which yields
$$ \dfrac{m}{m+1}(p+2+2\gamma) = \dfrac{2rm}{2m+r}(p-q+1) .$$
By virtue of the bound \eqref{gap}, we have that $\gamma >0$ for an exponent $m \in \mathbb{N}$ sufficiently large. Indeed,
$$\gamma >0 \Longleftrightarrow  2m[r(p-q+1)-(p+2)]>r(2q-p)$$
and
$$r(p-q+1)-(p+2)>0 \Longleftrightarrow q < p+1- \dfrac{p+2}{r}$$
that holds true by the assumption on the gap \eqref{gap}.

 \noindent Noting that
\begin{equation}
    \dfrac{m+1}{2m+r}\le \dfrac{1}{2}, \quad \forall m \in \mathbb{N}, \label{constant}
\end{equation}
the definition of $\gamma$ implies that
$$p+2+2\gamma =\dfrac{2r(m+1)}{2m+r}(p-q+1) \le r (p-q+1),$$
hence estimate \eqref{stima2} becomes
\begin{align}
    &\int_{B_s}  |Du_j|^\frac{2rm(p-q+1)}{2m+r}dx  \notag\\
     & \le C \dfrac{\left( \Vert 1+g \Vert_{L^r(B_R)}+\Vert D^2 \psi \Vert_{L^r(B_R)} \right)^\frac{2m}{m+1}}{(t-s)^\frac{2m}{m+1}}  \left(  \int_{B_t} |u_j|^{2m} dx\right)^\frac{1}{m+1} \notag\\
     & \ \ \ \ \cdot \left(  \int_{B_t}  (1+|Du_j|)^\frac{2rm(p-q+1)}{2m+r} dx\right)^\frac{m(r-2)}{r(m+1)}, \notag
\end{align}
where the constant $C$ is independent of $j$ and $m$. By Young's inequality with exponents $(\frac{r(m+1)}{m(r-2)},\frac{r(m+1)}{2m+r})$, we infer
\begin{align}
    \int_{B_s}  |Du_j|^\frac{2rm(p-q+1)}{2m+r}dx \le & \ \dfrac{1}{2} \int_{B_t}  |Du_j|^\frac{2rm(p-q+1)}{2m+r} dx +|B_R| \notag\\
     & \ + 2^\frac{m(r-2)}{2m+r}C^\frac{r(m+1)}{2m+r} \dfrac{\left( \Vert 1+g \Vert_{L^r(B_R)}+\Vert D^2 \psi \Vert_{L^r(B_R)} \right)^\frac{2mr}{2m+r}}{(t-s)^\frac{2mr}{2m+r}}  \left(  \int_{B_t} |u_j|^{2m} dx\right)^\frac{r}{2m+r} \notag\\
     \le & \ \dfrac{1}{2} \int_{B_t}  |Du_j|^\frac{2rm(p-q+1)}{2m+r} dx +|B_R| \notag\\
     & \ + C \dfrac{\left( \Vert 1+g \Vert_{L^r(B_R)}+\Vert D^2 \psi \Vert_{L^r(B_R)} \right)^\frac{2mr}{2m+r}}{(t-s)^r}  \left(  \int_{B_R} |u_j|^{2m} dx\right)^\frac{r}{2m+r}, \notag
\end{align}
where in the last inequality we used \eqref{constant}, the fact that $C \ge 1$ and that $R-\rho \le 1$. Eventually, an application of Lemma \ref{lm2} and using inequality \eqref{en6} yield the desired estimate.
\end{proof}

\noindent \textbf{Step 3.} We establish a uniform a priori estimate for the $L^\infty$-norm of the gradient of the approximating minimizers.

The proof of the a priori estimate relies on a classical Moser iteration argument, combined with the higher differentiability and the higher integrability results proved in Step 2. We have the following 

\begin{thm}
   Let $u_j \in W^{1,p}(\Omega') \cap L^{2m}(\Omega')$ be the solution to \eqref{OP}, under assumptions \eqref{eqf1}--\eqref{eqf5}. Assume that $2 \le p < q$ satisfy \eqref{gap} and $\psi \in W^{2, r}_{loc}(\Omega)$. Then, $u_j$ is locally Lipschitz continuous and the following inequality
   \begin{equation}\label{UnifEL}
       \Vert Du_j \Vert_{L^\infty(B_{R/2})} \le C \left(\dfrac{\Lambda}{R}\right) ^\frac{2 \cdot 2^*}{ \chi_m p\left(2^*- \frac{2r}{r-2}\right)}\left(  \dfrac{\Lambda^\frac{2rm}{2m+r}}{R^r} \left( 1+\Vert \psi \Vert_{L^\infty(\Omega')}^{2m}+a^{2m} \right)^\frac{r}{2m+r}\right)^\frac{r-2}{pr\chi_m}
   \end{equation}
   holds for every ball $B_R \Subset \Omega'$, where $\chi_m=\frac{2m}{2m+r}\frac{(r-2)(p-q+1)}{pr}-\frac{2 \cdot 2^*(q-p)}{p\left( 2^*-\frac{2r}{r-2}\right)}$ and $C=C(n,p,q,l,\nu, L, \tilde{L})$ is a positive constant independent of $j$ and $m$.
\end{thm}
The proof of the previous result goes as the one of \cite[Theorem 1.2, Step 1.]{EPdN}, taking into account that now the constant in the estimate \eqref{HighInt} depends also on the $L^r$-norm of $D^2 \psi$, and so it will not be presented here. 

\noindent \textbf{Step 4.} Now, we conclude
showing that the a priori estimate is preserved when passing to the limit. 

From Lemma \ref{lemmaL}, we have that 
$$u_j \to u \text{ \ strongly in \ } W^{1,p}(\Omega').$$
Therefore, taking the limit as $j \to \infty$ in \eqref{UnifEL}, we obtain
\begin{align}
    \Vert Du \Vert_{L^\infty(B_{R/2})} \le & \liminf_j  \Vert Du_j \Vert_{L^\infty(B_{R/2})} \notag\\
    \le & \ C \ \left(\dfrac{\Lambda}{R}\right) ^\frac{2 \cdot 2^*}{ \chi_m p\left(2^*- \frac{2r}{r-2}\right)}\left(  \dfrac{\Lambda^\frac{2rm}{2m+r}}{R^r} \left( 1+\Vert \psi \Vert_{L^\infty(\Omega')}^{2m}+a^{2m} \right)^\frac{r}{2m+r}\right)^\frac{r-2}{pr\chi_m}. \label{StimaUn}
\end{align}
Notice that
$$\lim_m \chi_m =\frac{(r-2)(p-q+1)}{pr}-\frac{2 \cdot 2^*(q-p)}{p\left( 2^*-\frac{2r}{r-2}\right)}=: \tilde{\chi}.$$
Hence, since \eqref{StimaUn} holds for every $m$ sufficiently large, we can pass to the limit as $m \to \infty$ and derive
$$ \Vert Du \Vert_{L^\infty(B_{R/2})} \le C(\Lambda,R)^{\tilde{\chi}}(1+\Vert \psi \Vert_{L^\infty(\Omega')}+a)^\beta,$$
where $\tilde{\chi}$ and $\beta$ depends only on $n,p,q$ and $r$. Finally, letting $a \to \Vert u-\psi \Vert_{L^\infty(\Omega')}$, we get the conclusion.

\vspace{1.5cm}

\noindent \textbf{Acknowledgements.} The authors are members of the Gruppo Nazionale per l’Analisi Matematica,
la Probabilità e le loro Applicazioni (GNAMPA) of the Istituto Nazionale di Alta Matematica (INdAM). The authors have been partially supported through the INdAM$-$GNAMPA 2024 Project “Interazione ottimale tra la regolarità dei coefficienti e l’anisotropia del problema in funzionali integrali a crescite non standard” (CUP: E53C23001670001). The work of the
first author was partially supported by the Universit\`a degli Studi di Napoli Parthenope through the project Bando ricerca locale 2023 - Sustainable Change: Towards a Society of Inclusion, Health and Green Economy.
\vspace{0.5cm}

\noindent \textbf{Author Contributions} All three authors contributed equally to this manuscript.

\vspace{0.5cm}

\noindent \textbf{Data availability statement} The authors declare all data supporting the findings of this study are available in the paper and in its references.

\vspace{1.5cm}

\noindent { \sc \textbf{ Declarations}}\\

\noindent \textbf{Conflict of interest} The authors have no relevant financial or non-financial interests to disclose. The authors
have no conflicts of interest to declare that are relevant to the content of this article.

\vspace{0.5cm}

{\sc Raffaella Giova} \\
DiSEG - Universit\`a degli Studi di Napoli ``Parthenope''\\
Via Generale Parisi, 13 - 80132 Napoli, Italy\\
\textit{Email address}: raffaella.giova@uniparthenope.it

\vspace{0.3cm}

{\sc Antonio Giuseppe Grimaldi} \\ 
Dipartimento di Ingegneria, Università degli Studi di Napoli ``Parthenope''\\
Centro Direzionale Isola C4 - 80143 Napoli, Italy\\
\textit{Email address}: antoniogiuseppe.grimaldi@collaboratore.uniparthenope.it

\vspace{0.3cm}

{\sc Andrea Torricelli} \\
Dipartimento di Scienze Matematiche ``G. L. Lagrange'' - Politecnico di Torino\\
Corso Duca degli Abruzzi, 24 - 10129 Torino, Italy\\
\textit{Email address}: andrea.torricelli@polito.it

\end{document}